\documentclass[12pt, leqno]{article}
\usepackage{fullpage} %comment this out for fat margins
\usepackage{amsmath}
\usepackage{amssymb}
\usepackage{undertilde}
\usepackage{stmaryrd}
\usepackage{bbding}
\usepackage{enumerate}
\usepackage{graphicx}
\usepackage{xypic}
\xyoption{curve}
\usepackage{multicol}
\xyoption{color}
\usepackage{setspace}
\usepackage{boxedminipage}
\usepackage{nicefrac}
\usepackage{multicol}
\usepackage{amsthm}
  \usepackage[pdftex,
  bookmarks = false,
  pdfstartview = FitBH,
  linktocpage = true,
  pagebackref = true,
  pdfhighlight = /O,
  colorlinks=true,
  linkcolor=blue,
  citecolor=blue,
  filecolor = blue,
  urlcolor = blue,
  menucolor = blue,
]{hyperref}

\newtheorem{thm}{Theorem}[section]
\newtheorem{prop}[thm]{Proposition}
\newtheorem{defn}[thm]{Definition}

\numberwithin{equation}{section}

\title{{F}ragments of Frege's \emph{{G}rundgesetze} and {G}\"odel's {C}onstructible {U}niverse
}
\author{Sean Walsh}

\begin{document}

\maketitle

\begin{abstract}
Frege's \emph{Grundgesetze} was one of the 19th century forerunners to contemporary set theory which was plagued by the Russell paradox. In recent years, it has been shown that subsystems of the \emph{Grundgesetze} formed by restricting the comprehension schema are consistent. One aim of this paper is to ascertain how much set theory can be developed within these consistent fragments of the \emph{Grundgesetze}, and our main theorem (Theorem~\ref{thm:main}) shows that there is a model of a fragment of the \emph{Grundgesetze} which defines a model of all the axioms of Zermelo-Fraenkel set theory with the exception of the power set axiom. The proof of this result appeals to G\"odel's constructible universe of sets and to Kripke and Platek's idea of the projectum, as well as to a weak version of uniformization (which does not involve knowledge of Jensen's fine structure theory). The axioms of the \emph{Grundgesetze} are examples of \emph{abstraction~principles}, and the other primary aim of this paper is to articulate a sufficient condition for the consistency of abstraction principles with limited amounts of comprehension (Theorem~\ref{thm:jointconsistency}). As an application, we resolve an analogue of the joint~consistency problem in the predicative setting.
\end{abstract}

\section{Introduction}

There has been a recent renewed interest in the technical facets of Frege's \emph{Grundgesetze} (\cite{Burgess2005},~\cite{Cook2007aa}) paralleling the long-standing interest in Frege's philosophy of mathematics and logic (\cite{Dummett1991a}, \cite{Beany2005aa}). This interest has been engendered by the consistency proofs, due to Parsons~\cite{Parsons1987a}, Heck~\cite{Heck1996}, and Ferreira-Wehmeier~\cite{Ferreira2002aa}, of this system with limited amounts of comprehension. The broader intellectual interest in Frege's \emph{Grundgesetze} stems in part from the two related ways in which it was a predecessor of contemporary set theory: first, the system was originally designed to be able to reconstruct much of ordinary mathematics, and second it comes equipped with the resources needed to define a membership relation. It is thus natural to ask how much set theory can be consistently developed within these fragments of the \emph{Grundgesetze}. Our main theorem (Theorem~\ref{thm:main}) shows it is possible within some models of these fragments to recover all the axioms of Zermelo-Fraenkel set theory with the exception of the power set axiom. To make this precise, one needs to carefully set out the primitives of the consistent fragments of the \emph{Grundgesetze} and indicate what precisely it means to recover a fragment of set theory. This is the primary goal of \S\ref{sec02.1} of the paper.

Following Wright and Hale (\cite{Hale2001}, cf.~\cite{Cook2007aa}), the system of the \emph{Grundgesetze} has been studied in recent decades as a special case of so-called \emph{abstraction principles}. These are principles that postulate lower-order representatives for equivalence relations on higher-order entities. Many of these principles are inconsistent with full comprehension, which intuitively says that every formula determines a concept or higher-order entity. So as with the \emph{Grundgesetze}, the idea has been to look for consistency with respect to the so-called predicative instances of the comprehension schema, in which the presence of higher-order quantifiers within formulas is highly restricted. Of course, while predicativity in connection with the \emph{Grundgesetze} is a fairly new topic, predicativity has a long tradition within mathematical logic, beginning with Poincar\'e, Russell, and Weyl (\cite{Heinzmann1986aa},~\cite{Weyl1918}), and found in our day in the work of Feferman (\cite{Feferman2005ab,Feferman1964aa}) and in~${\tt ACA}_0$ and related systems of Friedman and Simpson's project of reverse mathematics (\cite{Friedman1975aa},~\cite{Simpson2009aa}).

The other chief theorem of this paper (Theorem~\ref{thm:jointconsistency}) shows that an abstraction principle associated to an equivalence relation is consistent with predicative comprehension so long as the equivalence relation is provably an equivalence relation in a limited theory of pure second-order logic and is expressible in the signature of that pure second-order logic. One application of this result is a resolution to the joint consistency problem in the predicative setting. For, in the setting of full comprehension, it has been known for some time that there are abstraction principles which are individually but not jointly consistent. In \S\ref{sec0.1555predab}, we define the notion of an abstraction principle and further contextualize our results within the extant literature on abstraction principles.

The methods used in all these results draw on considerations related to G\"odel's constructible universe of sets. Whereas in the cumulative hierarchy of sets~$V_{\alpha}$, one proceeds by iterating the operation of the powerset into the transfinite, in the constructible hierarchy of sets~$L_{\alpha}$, one proceeds by iterating the operation of taking definable subsets into the transfinite. G\"odel showed that just like the universe of sets~$V=\bigcup_{\alpha} V_{\alpha}$ is a model of the axioms of set theory, so the constructible universe of sets~$L=\bigcup_{\alpha} L_{\alpha}$ is a model of the axioms of set theory, along with a strong form of the axiom of choice according to which the elements of~$L$ are well-ordered by a relation~$<_L$ (cf. \cite{Jech2003}~Chapter 13,~\cite{Kunen1980} Chapter 6,~\cite{Kunen2011aa} II.6, \cite{Devlin1984aa}). 

Our present understanding of the more ``local'' or ``micro'' properties of the constructible sets was furthered by the work of Kripke (\cite{Kripke1964aa}), Platek (\cite{Platek1966aa}) and Jensen (\cite{Jensen1972aa}), in whose results we find the key ideas of the projectum and uniformization. Roughly, a level~$L_{\alpha}$ of the constructible hierarchy satisfies uniformization if whenever it satisfies~$\forall \; x \; \exists \; y \; R(x,y)$ then there is a definable function~$f$ of the same level of complexity as~$R$ which satisfies~$\forall \;x \; R(x,f(x))$. The projectum, on the other hand, is related to the idea that certain initial segments~$L_{\alpha}$ of the constructible universe can be shrunk via a definable injection~$\iota:L_{\alpha}\rightarrow \rho$ to a smaller ordinal~$\rho<\alpha$. The formal definitions of the projectum and uniformization are given in \S\ref{sec02.2}. It bears emphasizing that we only employ a weak version of uniformization which has an elementary proof, and so this paper does not presuppose knowledge of Jensen's fine structure theory (cf. Proposition~\ref{prop:weakuniform}).

It's actually rather natural to think that uniformization and the projectum would be useful in producing models of abstraction principles. On the one hand, given an equivalence relation~$E$ on the set~$P(\rho)\cap L_{\alpha}$, we can conceive of the elements of this set as higher-order entities, and then we can take the lower-order representative in~$\rho$ of an~$E$-equivalence class to be the injection~$\iota$ applied to the~$<_L$-least element of~$E$'s equivalence class. On the other hand, uniformization allows one to secure further instances of the comprehension schema in which there are some controlled occurrences of higher-order quantifiers, in essence because one can use uniformization to choose one particular higher-order entity with which to work. 

This, in any case, is the intuitive idea behind the proof of our theorem on the consistency of abstraction principles (Theorem~\ref{thm:jointconsistency}) which we prove in \S\ref{sec04.5}. However, this does not itself deliver our result on how much set theory one can recover in the consistent fragments of the \emph{Grundgesetze}. For this, we need to additionally show that if we start from a level of the constructible hierarchy which satisfies certain axioms of set theory, and if we perform the construction of a model of the fragment of the \emph{Grundgesetze} in the manner intimated in the above paragraph and made precise in \S\ref{sec04.5}, then we can recover these original constructible sets definably within the model of the fragment of the \emph{Grundgesetze}. The details of this argument are carried out in \S\ref{sec05} where our Main Theorem~\ref{thm:main} is finally established.

This paper is the first in a series of three papers -- the other two being~\cite{Walsh2014ad},~\cite{Walsh2014ae}-- which collectively constitute a sequel to the Basic Law~V components of our paper~\cite{Walsh2012aa}. In that earlier paper~\cite{Walsh2012aa}, we gave a proof of the consistency of Frege's \emph{Grundgesetze} system with limited amounts of comprehension using tools from hyperarithmetic theory. However, we were unable to use these models to ascertain how much Zermelo-Fraenkel set theory could be consistently done in the Fregean setting. The work in this paper explains why this was the case. The key to this was an axiom known as Axiom~Beta (cf. Definition~\ref{defna:axiombeta}), which effectively ensures that the Mostowski Collapse Theorem holds in a structure. As one can see by inspection of the proofs in \S\ref{sec05}, it is being able to invoke this theorem in a model which allows us to obtain finally the Main Theorem~\ref{thm:main}. It turns out that the usual models associated to hyperarithmetic theory simply are not models of Axiom~Beta.

This present paper does not depend on results from our earlier paper~\cite{Walsh2012aa}, nor does it depend on its two thematically-linked companion papers,~\cite{Walsh2014ae}, ~\cite{Walsh2014ad}. In the companion paper~\cite{Walsh2014ae}, we use the constructible hierarchy to develop models of an intensional type theory, roughly analogous to how one can use the cumulative hierarchy to build models of an extensional type theory. This intensional type theory can in turn interpret fragments of the \emph{Grundgesetze} system, and so stands to the predicative \emph{Grundgesetze} system as the stage axioms of Shoenfield~\cite{Shoenfield1961aa, Shoenfield1967aa, Shoenfield1977aa} and Boolos~\cite{Boolos1971} stands to the Zermelo-Fraenkel system. In the other companion paper~\cite{Walsh2014ad}, we examine the deductive strength of the theory consisting of all the predicative abstraction principles whose consistency we establish here. 

\section{The \emph{Grundgesetze} and Its Set Theory}\label{sec02.1}

Basic Law~V is the crucial fifth axiom of Frege's \emph{Grundgesetze} (\cite{Frege1893},~\cite{Frege2013aa}), and it axiomatizes the behavior of a certain type-lowering operator from second-order entities to first-order entities, called the ``extension operator.'' In Frege's type-theory, the second-order entities are called ``concepts'' while the first-order entities are called ``objects,'' so that the extension operator~$\partial$ takes a concept~$X$ and returns an object~$\partial(X)$. (There is no standard notation for the extension operator, and so some authors write~$\S(X)$ in lieu of~$\partial(X)$). Basic Law~V then simply postulates that the extension operator is injective:
\begin{equation}
\mbox{\emph{Basic Law~V:}}\hspace{10mm}~\forall \; X, Y \; (\partial(X) = \partial(Y) \leftrightarrow X=Y)\label{eqn:blv}
\end{equation}
Here the identity of concepts is regarded as extensional in character, so that two concepts~$X,Y$ are said to be identical precisely when they are coextensive, i.e.~$X=Y$ if and only if for all objects~$z$ we have that~$Xz$ if and only if~$Yz$. 

Models of Basic Law~V have the following form:
\begin{equation}\label{eqn:blvmodels}
\mathcal{M}=(M, S_1(M), S_2(M), \ldots, \partial)
\end{equation}
wherein~$M$ is a non-empty set that serves as the interpretation of the objects, and the set~$S_n(M)\subseteq P(M^n)$ serves as the interpretation of the~$n$-ary concepts, and wherein the function~$\partial:S_1(M)\rightarrow M$ is an injection. Further, we assume that in the object-language of the structure from equation~(\ref{eqn:blvmodels}) we have the resources to describe when an~$n$-tuple~$(a_1, \ldots, a_n)$ from~$M^n$ is in an~$n$-ary concept~$R$ from~$S_n(M)$, and we write this in the object-language alternatively as~$R(a_1, \ldots, a_n)$ or~$(a_1, \ldots, a_n)\in R$, and we refer to this relation as the predication relation.

As is well-known, Basic Law~V is inconsistent with the full second-order comprehension schema:
\begin{defn} 
The \emph{Full Comprehension Schema} consists of the all axioms of the form \;$\exists \; R \; \forall \; \overline{a} \; (R\overline{a} \leftrightarrow \varphi(\overline{a}))$, wherein~$\varphi(\overline{x})$ is allowed to be any formula, perhaps with parameters, and~$\overline{x}$ abbreviates~$(x_1, \ldots, x_n)$ and~$R$ is an~$n$-ary concept variable for~$n\geq 1$ that does not appear free in~$\varphi(\overline{x})$.
\label{eqn:unaryfullcomp}
\end{defn}
\noindent In spite of this inconsistency, Parsons and Heck (\cite{Parsons1987a}, \cite{Heck1996}) showed that Basic Law~V is \emph{consistent} with the version of the comprehension schema in which~$\varphi(x)$ contains no second-order quantifiers:
\begin{defn}
The \emph{First-Order Comprehension Schema} consists of all axioms of the form \;$\exists \; R \; \forall \; \overline{a} \; (R\overline{a} \leftrightarrow \varphi(\overline{a}))$, wherein~$\varphi(\overline{x})$ is allowed to be any formula with no second-order quantifiers but perhaps with parameters, and~$\overline{x}$ abbreviates~$(x_1, \ldots, x_n)$ and~$R$ is an~$n$-ary concept variable for~$n\geq 1$ that does not appear free in~$\varphi(\overline{x})$. \label{pred:comp:schema}
\end{defn}
\noindent Ferreira and Wehmeier extended the Parsons-Heck result by showing that there are models~$\mathcal{M}=(M, D(M), D(M^2), \ldots, \partial)$ of Basic Law~V which also model stronger forms of comprehension, namely the \;$\Delta^1_1$-comprehension schema and the~$\Sigma^1_1$-choice schema (\cite{Ferreira2002aa} \S4). These schemata are defined as follows:
\begin{defn} The {\it$\Delta^1_1$-Comprehension Schema} consists of all axioms of the form
\begin{equation}
\forall \; \overline{x} \; (\varphi(\overline{x})\leftrightarrow \psi(\overline{x}))\rightarrow \exists \; R \; \forall \; \overline{a} \; (R\overline{a} \leftrightarrow \varphi(\overline{a}))
\end{equation}
wherein~$\varphi(\overline{x})$ is a~$\Sigma^1_1$-formula and~$\psi(\overline{x})$ is a~$\Pi^1_1$-formula that may contain parameters, and~$\overline{x}$ abbreviates~$(x_1, \ldots, x_n)$, and~$R$ is an~$n$-ary concept variable for~$n\geq 1$ that does not appear free in~$\varphi(\overline{x})$ or~$\psi(\overline{x})$. \label{delta11comp}
\end{defn}
\begin{defn} 
The {\it$\Sigma^1_1$-Choice Schema} consists of all axioms of the form
\begin{equation}\label{eqn:LATE:formchoice}
[\forall \; \overline{x} \; \exists \; R^{\prime} \; \varphi(\overline{x}, R^{\prime})]\rightarrow \exists \; R \; [\forall \; \overline{x} \; \forall \; R^{\prime} \; [(\forall \; \overline{y} \; (R^{\prime}\overline{y} \leftrightarrow R\overline{x}\:\overline{y}))\rightarrow \varphi(\overline{x}, R^{\prime})]]
\end{equation}
wherein the formula~$\varphi(\overline{x},R^{\prime})$ is~$\Sigma^1_1$, perhaps with parameters, and~$\overline{x}$ abbreviates {}$(x_1, \ldots, x_n)$ and~$\overline{y}$ abbreviates~$(y_1, \ldots, y_m)$ and~$R$ is an~$(n+m)$-ary concept variable for~$n,m\geq 1$ that does not appear free in~$\varphi(\overline{x},R^{\prime})$ where~$R^{\prime}$ is an~$m$-ary concept variable.
\label{sigam11choice}
\end{defn}
\noindent Here, as is usual, a~$\Sigma^1_1$-formula (resp.~$\Pi^1_1$-formula) is one which begins with a block of existential quantifiers (resp. universal quantifiers) over~$n$-ary concepts for various~$n\geq 1$ and which contains no further second-order quantifiers. Given this variety of comprehension schemata, it becomes expedient to explicitly distinguish between different formal theories that combine these schemata with the axiom Basic Law~V from equation~(\ref{eqn:blv}). In particular, one defines the following systems  (cf.~\cite{Walsh2012aa} Definition 5 p. 1683):
\begin{defn} The theory~$\tt{ ABL}_0$ is Basic Law~V together with the First-Order Comprehension Schema~(cf. Definition~\ref{pred:comp:schema}). The theory~${\tt \Delta^1_1\mbox{-}BL_0}$ is Basic Law~V together with the~$\Delta^1_1$-Comprehension Schema (cf. Definition~\ref{delta11comp}).  The theory~${\tt \Sigma^1_1\mbox{-}LB_0}$ is Basic Law~V together with the~$\Sigma^1_1$-Choice Schema (cf. Definition~\ref{sigam11choice}) and the First-Order Comprehension Schema~(cf. Definition~\ref{pred:comp:schema}).\label{defn:Sigma11choiceBLV}
\end{defn}
\noindent We opt to designate the subsystem formed with~$\Sigma^1_1$-Choice by inverting the letters ``{\tt BL}'' to ``{\tt LB}'', since this convention saves us from needing to write out the word ``choice'' when referring to a theory, and since it is compatible with the convention in subsystems of second-order arithmetic (\cite{Simpson2009aa}), wherein the~$\Delta^1_1$-comprehension fragment is called~${\tt \Delta^1_1\mbox{-}CA_0}$ and the~$\Sigma^1_1$-choice fragment is called~${\tt \Sigma^1_1\mbox{-}AC_0}$. 

In the companion paper \cite{Walsh2014ad}, we work deductively in theories containing limited amounts of comprehension. In these situations, it will prove expedient to consider an enrichment of the above theories by the addition of certain function symbols. In particular, we assume that for every~$m,n>0$ we have a~$(n+1)$-ary function symbol in the language for the map~$(R, a_1, \ldots, a_n)\mapsto R[a_1, \ldots, a_n]$ from a single~$(n+m)$-ary relation~$R$ and an~$n$-tuple of objects~$(a_1, \ldots, a_n)$ to the~$m$-ary relation
\begin{equation}\label{eqn:iamafunction}
R[a_1, \ldots, a_n]= \{(y_1, \ldots, y_m): R(a_1, \ldots, a_n, y_1, \ldots, y_m)\}
\end{equation}
One benefit of the addition of these symbols is that it allows for a compact formalization of the key clause~(\ref{eqn:LATE:formchoice}) of the $\Sigma^1_1$-choice schema, namely:
\begin{equation}\label{eqn:LATE:formchoice2}
[\forall \; \overline{x} \; \exists \; R^{\prime} \; \varphi(\overline{x}, R^{\prime})]\rightarrow [\exists \; R \; \forall \; \overline{x} \; \varphi(\overline{x}, R[\overline{x}])]
\end{equation}
The addition of these function symbols to the signature impacts the axiom system because we continue to assume that we have~$\Sigma^1_1$-choice and first-order comprehension. In particular, the inclusion of the function symbols~$(R, a_1, \ldots, a_n)\mapsto R[a_1, \ldots, a_n]$ in the signature then adds to the collection of terms of the signature, which in turn adds to the collection of quantifier-free and hence first-order formulas of the signature. 

Let us call this expansion of~${\tt \Sigma^1_1\mbox{-}LB_0}$ the system~${\tt \Sigma^1_1\mbox{-}LB}$, i.e. we drop the ``zero'' subscript; and likewise for the other systems from Definition~\ref{pred:comp:schema}. For ease of future reference, let's explicitly record this in the following definition:
\begin{defn} The theory~$\tt{ ABL}$ is Basic Law~V together with the First-Order Comprehension Schema~(cf. Definition~\ref{pred:comp:schema}) in the signature including the function symbols~$(R, a_1, \ldots, a_n)\mapsto R[a_1, \ldots, a_n]$. The theory~${\tt \Delta^1_1\mbox{-}BL}$ is Basic Law~V together with the~$\Delta^1_1$-Comprehension Schema (cf. Definition~\ref{delta11comp}) in the signature with these function symbols. The theory~${\tt \Sigma^1_1\mbox{-}LB}$ is Basic Law~V together with the~$\Sigma^1_1$-Choice Schema (cf. Definition~\ref{sigam11choice}) and the First-Order Comprehension Schema~(cf. Definition~\ref{pred:comp:schema}) in the signature including these function symbols.\label{defn:Sigma11choiceBLV00}
\end{defn}

In building models of these consistent fragments of Frege's system, one of our chief aims is to understand how much set theory can be thereby recovered. The crucial idea is to define an ersatz membership-relation~$\eta$ in terms of the extension operator and predication:
\begin{equation}\label{eqn:Fregemembership}
a\eta b \Longleftrightarrow \exists \; B \; (\partial(B) =b \; \& \; Ba)
\end{equation}
Since the \emph{extensions} are precisely the objects in the range of the extension operator~$\partial$, we write the collection of extensions as~$\mathrm{rng}(\partial)$. Now it follows from considerations related to the Russell paradox that~$\mathrm{rng}(\partial)$ is not a concept in the presence of~$\Delta^1_1$-comprehension (cf. \cite{Walsh2012aa} Proposition 29 p. 1692). In contrast to~$\mathrm{rng}(\partial)$, the collections~$V=\{x:x=x\}$ and~$\emptyset =\{x: x\neq x\}$ do form concepts since they are first-order definable. The following elementary proposition, provable in~${\tt \Sigma^1_1\mbox{-}LB}$, says that for subconcepts of~$\mathrm{rng}(\partial)$, the~$\eta$-relation restricted to this concept exists as a binary concept:
\begin{prop}\label{iamaveryhelpfulprop} (Existence of Restricted~$\eta$-relation) (${\tt \Sigma^1_1\mbox{-}LB}$)
For every concept~$X\subseteq \mathrm{rng}(\partial)$ there is a binary concept~$R$ such that for all~$a$, we have that~$Xa$ implies~$\partial(R[a])=a$. So for all concepts~$X\subseteq \mathrm{rng}(\partial)$ there is a binary relation~$E_X\subseteq V\times X$ such that~$Xa$ implies:~$E_X(b,a)$ iff~$b\eta a$.
\end{prop}

It will also be helpful in what follows to have some fixed notation for subset and successor. So similar to equation~(\ref{eqn:Fregemembership}) we define the associated Fregean subset relation~$\subseteq_{\eta}$ as follows:
\begin{equation} \label{eqn:defn:subsetfrege}
a\subseteq_{\eta} b \Longleftrightarrow \forall \; c \; (c\eta a \rightarrow c\eta b)
\end{equation}
However, note that if~$a$ is \emph{not} an extension, then~$c\eta a$ is always false and so~$(c\eta a\rightarrow \psi)$ is always true, regardless of what~$\psi$ is. Hence, if~$a$ is \emph{not} an extension, then~$a\subseteq_{\eta} b$ is always true. So the expressions~$a\eta b$ and~$a\subseteq_{\eta} b$ will behave like membership and subset only if one restricts attention to~$a,b,$ that are extensions. In what follows, it will also be useful to introduce some notation for a successor-like operation on extensions. So let us say that 
\begin{equation}\label{eqn:iamsuccessorfunctiondefin}
\sigma(x) = y \Longleftrightarrow \exists \; F \; \exists \; G \; [\partial(F)=x \; \& \; \partial(G)=y \; \& \; \forall \; z \; (Gz\leftrightarrow (Fz \vee z=x))]
\end{equation}
However, this function is not total, and in particular it should be emphasized that~$\sigma(x)$ is only well-defined when~$x$ is an extension. Accordingly, the graph of the function~$x\mapsto \sigma(x)$ does not exist as a binary concept, since if it did, then its domain would likewise exist, and its domain is precisely~$\mathrm{rng}(\partial)$. However, when~$\sigma(x)$ is defined, note that it satisfies~$z\eta (\sigma(x))$~iff either~$z\eta x$ or~$z=x$. This of course reminds us of the usual set-theoretic successor operation~$x\mapsto (x\cup \{x\})$.

In the axiomatic development of systems related to~${\tt \Sigma^1_1\mbox{-}LB}$, the crucially important concept is the notion of transitive closure. If~$F$ is a concept, then let us say that~$F$ is {\it~$\eta$-transitive} or {\it~$\eta$-closed} if~$(Fx \; \& \; y\eta x)$ implies~$Fy$, for all~$x,y$. Then we define transitive closure as follows:
\begin{equation}\label{eqn:iamtransitiveclosureeta} 
(\mathrm{Trcl}_{\eta}(x))(y) \equiv \forall \; F \; [\mbox{$F$ is~$\eta$-transitive \&~$x\subseteq_{\eta} \partial(F)$}]\rightarrow Fy
\end{equation}
It is easily provable that~$\mathrm{Trcl}_{\eta}(x)$ also has the following properties:
\begin{prop}\label{prop:elementary} (Elementary Facts about Transitive Closure)
\begin{enumerate}
\item \emph{Transitive Closure is~$\eta$-transitive}:~$[(\mathrm{Trcl}_{\eta}(x))(y) \wedge~z\eta y]$ implies~$(\mathrm{Trcl}_{\eta}(x))(z)$.
\item \emph{Transitive Closure is an~$\eta$-superclass}:~$w\eta x$ implies~$(\mathrm{Trcl}_{\eta}(x))(w)$.
\end{enumerate}
\end{prop}

So now we may describe the procedure for carving out a model of a fragment of classical set theory~{\tt ZFC} from a model of~$\mathcal{M}$ of~${\tt \Sigma^1_1\mbox{-}LB}$. Since the foundation~axiom is a traditional part of~{\tt ZFC}, we want to ensure that our fragments always include this axiom, and for this purpose it is important that we avoid infinite descending~$\eta$-chains. Since~$\mathcal{M}$ has second-order resources, this can be effected in a straightforward manner. In particular, if~$\mathcal{X}\subseteq M$ and~$\mathcal{R}\subseteq \mathcal{X}\times \mathcal{X}$ are~$\mathcal{M}$-definable (but not necessarily elements of~$S_k(M)$), then let us say that ``{\it{}$(\mathcal{X},\mathcal{R})$ is well-founded in~$\mathcal{M}$}'' if~$\mathcal{M}$ models that every non-empty \emph{subconcept} of~$\mathcal{X}$ has an~$\mathcal{R}$-least member, i.e.~$\mathcal{M}$ models~$\forall \; F \; [\exists \; x \; Fx \; \& \; \forall \; x \; (Fx\rightarrow \mathcal{X}(x))] \rightarrow [\exists \; y \; Fy \; \& \; \forall \; z \; (Fz \rightarrow \neg \mathcal{R}(z,y))]$. A special case of this is when~$X$ is a concept and~$R$ is a binary concept, in which case we likewise define ``\emph{$(X,R)$ is well-founded in~$\mathcal{M}$}'' to mean that~$\mathcal{M}$ models that every non-empty subconcept of~$X$ has an~$R$-least element, i.e that~$\mathcal{M}$ models
\begin{equation}\label{eqn:well-foundeddd123412341}
\forall \; F \; [\exists \; x \; Fx \; \& \; \forall \; x \; (Fx\rightarrow Xx)] \rightarrow [\exists \; y \; Fy \; \& \; \forall \; z \; (Fz \rightarrow \neg R(z,y))]
\end{equation}
Since~$S_1(M)$ is in general a small subset of~$P(M)$, we of course need to be wary of inferring from ``$(X, R)$ is well-founded in~$\mathcal{M}$'' to~$(X, R)$ having no infinite descending~$R$-chains, or to~$(X,R)$ having no infinite~$\mathcal{M}$-definable descending~$R$-chains.

Finally, putting this all together, let us define the notion of a ``well-founded extension'':
\begin{align}
\mathrm{wfExt}(x) \equiv x\mbox{ is an extension} & \; \& \;  (\mathrm{Trcl}_{\eta}(\sigma(x)), \eta) \mbox{ is well-founded} \notag \\
&  \; \& \; (\mathrm{Trcl}_{\eta}(\sigma(x)) \subseteq \mathrm{rng}(\partial)\label{eqnasdfsadf2}
\end{align}
Given a model~$\mathcal{M}$ of~${\tt \Sigma^1_1\mbox{-}LB}$, let us define its collection of well-founded extensions as follows:
\begin{equation}\label{eqnasdfsadf1}
\mathrm{wfExt}(\mathcal{M}) = \{x\in M: \mathcal{M}\models \mathrm{wfExt}(x)\}
\end{equation}
In broad analogy with its usage in set theory, we shall sometimes refer to this as the \emph{inner model} of well-founded extensions relative to a model of~${\tt \Sigma^1_1\mbox{-}LB}$.

The other definition that we need in order to state and prove our results is a global choice principle. Suppose that~$T$ is a theory in one of our signatures. Then we let~$T+{\tt GC}$ be the expansion of~$T$ by a new binary relation symbol~$<$ on objects in the signature, with axioms saying that~$<$ is a linear order of the first-order objects, and we additionally have a schema in the expanded signature saying that any instantiated formula~$\varphi(x)$ in the expanded signature, perhaps containing parameters, that holds of some first-order object~$x$ will hold of a~$<$-least element: 
\begin{equation}
[\exists \; x \; \varphi(x)]\rightarrow [\exists \; x \; \varphi(x) \; \& \; \forall \; y<x \; \neg \varphi(y)] \label{eqn:gcschema}
\end{equation}
Since all our theories~$T$ contain first-order comprehension~(cf. Definition~\ref{pred:comp:schema}), and since instances of~$<$ are quantifier-free and hence first-order, we have that the graph of~$<$ forms a binary concept in~$T+{\tt GC}$. Of course the postulated binary relation~$<$ does not necessarily have anything to do with the usual ``less than'' relation on the natural numbers.

With this all notation in place, our main theorem can be expressed as follows,  wherein~${\tt P}$ denotes the power set axiom:
\begin{thm}\label{thm:main} (Main Theorem)
There is a model~$\mathcal{M}$ of~${\tt \Sigma^1_1\mbox{-}LB}+{\tt GC}$ such that~$(\mathrm{wfExt}(\mathcal{M}), \eta)$ satisfies the axioms of~${\tt ZFC\mbox{-}P}$.
\end{thm}
\noindent This result is proven at the close of \S\ref{sec05}. It is significant primarily because it shows us what kind of set theory may be consistently developed if one takes Basic Law~V as a primitive. Now, one subtlety should be mentioned here at the outset: in the absence of power set, it is not entirely obvious which form of replacement and which form of choice is optimal. The discussion in Gitman-Hamkins-Johnstone (\cite{Gitman2011aa}) suggests that instead of the replacement schema one should use the collection schema, and as for the axiom of choice one should use the principle that every set can be well-ordered; the reason in each case being that these are the deductively stronger principles in the absence of powerset. (For a formal statement of the collection schema, cf. equation~(\ref{eqn:collectionschema})). As we will note when establishing our main theorem in \S\ref{sec05}, our models satisfy these principles as well. Hence, for the sake of concreteness, in this paper we may define~${\tt ZFC\mbox{-}P}$ as follows:
\begin{defn}\label{defn:ZFCminusP}
${\tt ZFC\mbox{-}P}$ is the theory consisting of extensionality, pairing, union, infinity, separation, collection, foundation, and the statement that every set can be well-ordered.
\end{defn}
\noindent For precise definitions of these axioms, one may consult any standard set theory textbook (\cite{Kunen1980,Kunen2011aa},~\cite{Jech2003}; and for the collection schema see again equation~(\ref{eqn:collectionschema})).

The Main Theorem~\ref{thm:main} is a natural analogue of the work of Boolos, Hodes, and Cook's on the axiom ``New V'' (\cite{Boolos1989aa},~\cite{Hodes1991},~\cite{Cook2003ab}). This is the axiom in the signature of Basic Law~V, but where, for the sake of disambiguation, we write the type-lowering operator with the symbol~$\partial^{\prime}$ as opposed to~$\partial$. The axiom \emph{New~V} then says that
\begin{equation}
 \emph{New~V}: \hspace{3mm} \forall \; X,Y \; (\partial^{\prime}(X)=\partial^{\prime}(Y) \leftrightarrow ((\mathrm{Small}(X) \vee \mathrm{Small}(Y)) \rightarrow X=Y)) \label{eqn:NewV}
\end{equation}
Here~$\mathrm{Small}(X)$ is an abbreviation for the statement that~$X$ is not bijective with the universe of first-order objects~$\{x: x=x\}$. So if~$\mathcal{M}=(M, S_1(M), S_2(M), \ldots, \partial^{\prime})$ is a model of New~V, then~$\mathcal{M}\models \mathrm{Small}(X)$ if and only if there's no bijection~$f:X\rightarrow M$ whose graph is in~$S_2(M)$. 

To see the connection between New~V and {\tt ZFC}, recall that for a cardinal~$\kappa$, the set~$H_{\kappa}$ is defined as~$H_{\kappa}=\{x: \left|\mathrm{trcl}(x)\right|<\kappa\}$ (cf.~\cite{Kunen1980} \S{IV.6} pp. 130 ff,~\cite{Kunen2011aa} p. 78,~\cite{Jech2003} p. 171). Suppose that~$\kappa>\omega$ is regular and satisfies~$\left|H_{\kappa}\right|=\kappa$. In this circumstance, let us define:
\begin{equation}
\mathbb{H}_{\kappa} = (H_{\kappa}, P(H_{\kappa}), P(H_{\kappa} \times H_{\kappa}), \ldots, \partial^{\prime})
\end{equation}
where~$\partial^{\prime}(X)=\langle 1, X\rangle$ if~$\left|X\right|<\kappa$ and~$\partial^{\prime}(X) = \langle 0,0\rangle$ otherwise (wherein~$\langle \cdot,\cdot\rangle$ is the usual set-theoretic pairing function). Then in analogue to Frege's definition of membership in equation~(\ref{eqn:Fregemembership}), we can define a quasi-membership relation~$\eta^{\prime}$ in models of New~V as follows:
\begin{equation}\label{eqn:newVmembership}
a\eta^{\prime} b \Longleftrightarrow \exists \; B \; (\mathrm{Small}(B) \; \& \; \partial^{\prime}(B)=b \; \& \; Ba)
\end{equation}
Likewise, we can define~$\mathrm{wfExt^{\prime}}$ using the relation~$\eta^{\prime}$ just as~$\mathrm{wfExt}$ is defined in equation~(\ref{eqnasdfsadf2}) using the relation~$\eta$. Then one may prove that~$\mathbb{H}_{\kappa}$ is a model of New~V and~$(\mathrm{wfExt^{\prime}}(\mathbb{H}_{\kappa}), \eta^{\prime})$ is isomorphic to~$(H_{\kappa}, \in)$, which is known to model~${\tt ZFC\mbox{-}P}$ when~$\kappa>\omega$ is regular (cf.~\cite{Kunen1980} Theorem IV.6.5 p. 132,~\cite{Kunen2011aa} Theorem II.2.1 p. 109,~\cite{Jech2003} p. 171). Hence one has following:
\begin{prop}\label{cor:mainanalogue}
There is a model~$\mathcal{M}$ of New~V and the Full Comprehension Schema~(cf. Definition~\ref{eqn:unaryfullcomp}) such that~$(\mathrm{wfExt^{\prime}}(\mathcal{M}), \eta^{\prime})$ satisfies the axioms of~${\tt ZFC\mbox{-}P}$ (cf. Definition~\ref{defn:ZFCminusP}).
\end{prop}
\noindent The Main Theorem~\ref{thm:main} establishes an analogous result for Basic~Law~V in the setting of limited amounts of comprehension.

\section{Predicative Abstraction Principles}\label{sec0.1555predab}

The axioms Basic Law~V and New~V are examples of what are now called \emph{abstraction principles}. If~$E(R,S)$ is a formula of second-order logic with exactly two free~$n$-ary relation variables for some~$n\geq 1$ then the \emph{abstraction principle}~$A[E]$ associated to~$E$ is the following axiom in a signature expanded by a new function symbol~$\partial_E$ from~$n$-ary relations to objects:
\begin{equation}
A[E]: \hspace{15mm} \forall \; R, S, [\partial_E(R)= \partial_E(S) \leftrightarrow E(R,S)]\label{defnAe}
\end{equation}
Abstraction principles have been studied extensively for many decades. For an introduction to this subject, see Burgess~\cite{Burgess2005}, and for many important papers, see the collections edited by Demopoulos~\cite{Demopoulos1995} and Cook~\cite{Cook2007aa}.

The first thing that one observes in this subject is that some abstraction principles are consistent with the Full Comprehension Schema in their signature~(cf. Definition~\ref{eqn:unaryfullcomp}) while others are not. For instance, we saw above that New~V is consistent with the Full Comprehension Schema in its signature, while Basic Law~V~(\ref{eqn:blv}) itself is not. Given that Basic Law~V is consistent with weaker forms of comprehension, one may ask whether there is any general method for determining whether the abstraction principle~$A[E]$ is consistent with these weaker forms of comprehension. In answering this question, it's helpful to have specific names for the theories consisting of combinations of the abstraction principle~$A[E]$ with the weaker forms of comprehension:
\begin{defn} \label{defn:oneabstract}
 For each formula~$E(R,S)$ with exactly two free~$n_E$-ary variables~$R,S$ for a specific~$n_E\geq 1$, let the theory~${\tt \Delta^1_1\mbox{-}A[E]}$ (resp.~${\tt \Sigma^1_1\mbox{-}[E]A}$) consist of~$A[E]$ from equation~(\ref{defnAe}) plus the~$\Delta^1_1$-Comprehension Schema (cf. Definition~\ref{delta11comp}) in the signature containing the function symbol~$\partial_E$ (resp.~$A[E]$ from equation~(\ref{defnAe}) plus the~$\Sigma^1_1$-Choice Schema (cf. Definition~\ref{sigam11choice}) and the First-Order Comprehension Schema~(cf. Definition \ref{pred:comp:schema}) in the signature containing the function symbol~$\partial_E$). 
\end{defn}
\noindent Further, let us define a theory of pure second-order logic in a very limited signature:
\begin{defn}
The theory~${\tt SO}$ is the second-order theory consisting of the Full Comprehension Schema~(\ref{eqn:unaryfullcomp}) in the signature of pure second-order logic bereft of all type-lowering function symbols. \label{eqn:secondorderlogic}
\end{defn}
\noindent Here the abbreviation~``${\tt SO}$'' is chosen because it reminds us of ``second-order logic.'' It's worth emphasizing that while the theory~${\tt SO}$ has full comprehension in its signature, its signature is very impoverished and does not include any of the type-lowering function symbols featuring in abstraction principles. But as with the fragments of Basic Law~V discussed in the previous sections, we're assuming that we have the function symbols~$(\overline{a}, R)\mapsto R[\overline{a}]$ from equation~(\ref{eqn:iamafunction}) in the signature of all our theories, including ${\tt SO}$. So just to be clear: the signature of ${\tt SO}$ consists merely of the predication relations $R\overline{x}$ and the maps~$(\overline{a}, R)\mapsto R[\overline{a}]$, and its axioms consist solely of the extensionality axioms and the instances of the Full Comprehension Schema~(\ref{eqn:unaryfullcomp}) in its signature. Sometimes in what follows we consider the extension ${\tt SO}+{\tt GC}$, which per the discussion of global choice in the previous section adds to the signature of ${\tt SO}$ a binary relation on objects and posits that it is a well-order. In the theory ${\tt SO}+{\tt GC}$, we adopt the convention that instances of the Full Comprehension Schema~(\ref{eqn:unaryfullcomp}) may include the global well-order.

One of our chief results on the consistency of abstraction principles with predicative levels of comprehension is the following:
\begin{thm}\label{cor:consistency}
Suppose that~$n\geq 1$ and that~$E(R,S)$ is a formula in the signature of~${\tt SO}+{\tt GC}$ which is provably an equivalence relation on~$n$-ary concepts in~${\tt SO}+{\tt GC}$. Then~${\tt \Sigma^1_1\mbox{-}[E]A}+{\tt SO}+{\tt GC}$ is consistent.
\end{thm}
\noindent \noindent This result is proven in \S\ref{sec04.5} below. This result indicates that the fact that Basic Law~V is consistent with the~$\Delta^1_1$-comprehension schema and~$\Sigma^1_1$-choice schema is not an isolated phenomena, but follows from the fact that~$E(X,Y)\equiv X=Y$ is provably an equivalence relation in~${\tt SO}+{\tt GC}$. It's worth stressing that the theory $T={\tt \Sigma^1_1\mbox{-}[E]A}+{\tt SO}+{\tt GC}$ has full comprehension for formulas in the signature of ${\tt SO}+{\tt GC}$ but only has predicative comprehension for formulas in~$T$'s full signature which includes the type-lowering function~$\partial_E$.

A related problem of long-standing interest has been the ``joint consistency problem.'' This is the problem of determining natural conditions on~$E_1, E_2$ so that if~$A[E_1]$ and~$A[E_2]$ has a standard model then~$A[E_1]\wedge A[E_2]$ has a standard model. A second-order theory is said to have a \emph{standard model} if it has a model~$\mathcal{M}$ satisfying~$S_n(M)=P(M^n)$ for all~$n\geq 1$, where we here employ the notation introduced in the previous section in equation~(\ref{eqn:blvmodels}) for models. This is a non-trivial problem: for, some~$A[E_1]$ have standard models~$\mathcal{M}$ only when the underlying first-order domain~$M$ is finite, such as when~$E_1(X,Y)$ is expressive of the symmetric difference of~$X$ and~$Y$ being Dedekind-finite (cf.~\cite{Boolos1998} p. 215,~\cite{Hale2001} pp. 289 ff). However, other~$A[E_2]$ have a standard model~$\mathcal{M}$ with underlying first-order domain~$M$ only when~$M$ is infinite, such as when~$E_2(X,Y)$ is expressive of~$X,Y$ being bijective. 

In the setting of limited amounts of comprehension, the most obvious analogue of the joint consistency problem is to ask about the extent to which it is consistent that~$A[E_1]\wedge A[E_2]$ has a model satisfying e.g. the~$\Delta^1_1$-comprehension schema when each~$A[E_i]$-individually does. Formally, let us introduce the following theories:
\begin{defn} 
 \label{predabstracdefn} The theory~${\tt \Delta^1_1\mbox{-}A[E_1, \ldots, E_k]}$ (resp.~${\tt \Sigma^1_1\mbox{-}[E_1, \ldots, E_k]A}$) consists both of the abstraction principles~$A[E_1] \wedge \cdots \wedge A[E_k]$~(\ref{defnAe}) and the~$\Delta^1_1$-Comprehension Schema (cf. Definition~\ref{delta11comp})  (resp. plus the~$\Sigma^1_1$-Choice Schema (cf. Definition~\ref{sigam11choice}) and the First-Order Comprehension Schema (cf. Definition~\ref{pred:comp:schema})) in the signature containing all the function symbols~$\partial_{E_1}, \ldots, \partial_{E_k}$.
\end{defn}

Our result Theorem~\ref{cor:consistency} from above is a direct consequence of the following theorem, which indicates that the joint consistency problem does not arise in the setting with limited amounts of comprehension, assuming that we can prove the formulas are equivalence relations in~${\tt SO}+{\tt GC}$, and assuming that the equivalence relations are expressible in the signature of~${\tt SO}+{\tt GC}$:
\begin{thm}\label{thm:jointconsistency} (Joint Consistency Theorem) 
Suppose~$n_1, \ldots, n_k\geq 1$ and that the formulas~$E_1(R,S)$, \ldots,~$E_k(R,S)$ in the signature of~${\tt SO}+{\tt GC}$ are provably equivalence relations on~$m_i$-ary concepts in~${\tt SO}+{\tt GC}$. Then the theory~${\tt \Sigma^1_1\mbox{-}[E_1, \ldots, E_k]A}+{\tt SO}+{\tt GC}$ is consistent.
\end{thm}
\noindent This result is proven in \S\ref{sec04.5} below. It's worth again underscoring that the theory $T={\tt \Sigma^1_1\mbox{-}[E_1, \ldots, E_k]A}+{\tt SO}+{\tt GC}$ has full comprehension for formulas in the signature of ${\tt SO}+{\tt GC}$ but only has predicative comprehension for formulas in~$T$'s full signature which includes the type-lowering function~$\partial_{E_1}, \ldots, \partial_{E_k}$. By compactness, this theorem establishes the consistency of a theory which includes abstraction principles associated to each formula in the signature of ${\tt SO}+{\tt GC}$ which one can prove to be an equivalence relation in ${\tt SO}+{\tt GC}$. In our companion paper  \cite{Walsh2014ad}, we study the deductive strength of this theory.

\section{Constructibility and Generalized Admissibility}\label{sec02.2}

The aim of this section is to briefly review several of the tools from constructibility that we use in the below proofs. Hence, it might be advisable to skip this section on a first read-through and refer back to this section as needed. In this section, we work entirely with fragments and extensions of the standard {\tt ZFC}-set theory, so that all structures~$M$ are structures in the signature of set-theory. The tools which we review and describe in this section come from constructibility, the study of G\"odel's universe~$L$  (cf. ~\cite{Jech2003}~Chapter 13,~\cite{Kunen1980} Chapter 6,~\cite{Kunen2011aa} II.6, \cite{Devlin1984aa}). This is the union of the  sets~$L_{\alpha}$ that are defined recursively as follows, wherein~$\mathrm{Defn}(M)$ refers to the subsets of~$M$ which are definable with parameters (when~$M$ is conceived of as having, as its only primitive, the membership relation restricted to its elements):
\begin{equation}\label{defn:L}
L_0=\emptyset, \hspace{10mm} L_{\alpha+1}=\mathrm{Defn}(L_{\alpha}), \hspace{10mm} L_{\alpha}=\bigcup_{\beta<\alpha} L_{\beta} \mbox{~for~$\alpha$ a limit}
\end{equation}

One tool which we shall use frequently in this paper is the following natural generalization of the notion of an admissible ordinal:
\begin{defn}\label{defn:nadmissible}
For~$n\geq 1$, an ordinal~$\alpha$ is \emph{$\Sigma_n$-admissible} if~$\alpha$ a limit and~$\alpha>\omega$ and~$L_{\alpha}$ models~$\Sigma_n$-collection and~$\Sigma_{n-1}$-separation.
\end{defn}
\noindent Recall that the collection schema is the following schema: 
\begin{equation}\label{eqn:collectionschema}
\forall \; \overline{p} \; [\forall \; x \; \exists \; y \; \varphi(x,y,\overline{p}) ]\rightarrow [\forall \; u \; \exists \; v \; (\forall \; x\in u \; \exists \; y\in v \; \varphi(x,y,\overline{p}))]
\end{equation}
\noindent By abuse of notation, we also say that~$L_{\alpha}$ is~$\Sigma_n$-admissible iff~$\alpha$ is~$\Sigma_n$-admissible; and we write ``admissible'' in lieu of ``$\Sigma_1$-admissible.'' The notion of~$\Sigma_n$-admissibility can be described axiomatically as well. In particular, it is not difficult to see that~$L_{\alpha}$ is~$\Sigma_n$-admissible if and only if~$L_{\alpha}$ satisfies extensionality, pairing, union, infinity, the foundation schema,~$\Sigma_n$-collection, and~$\Sigma_{n-1}$-separation. In the case~$n=1$, this set of axioms provides an equivalent axiomatization of Kripke-Platek set theory (\cite{Kripke1964aa}, \cite{Platek1966aa}, \cite{Devlin1984aa} p. 48, p. 36). Further, the union of this set of axioms for all~$n\geq 1$, along with the axiom choice (in the form that every set can be well-ordered), is deductively equivalent to~${\tt ZFC\mbox{-}P}$ (cf. Definition~\ref{defn:ZFCminusP}). Finally, an equivalent definition of~$\Sigma_n$-admissibility is as follows:~$\alpha$ is~$\Sigma_n$-admissible if and only if~$\alpha$ is a limit and~$\alpha>\omega$ and~$L_{\alpha}$ models Kripke-Platek set theory and~$\Sigma_n$-replacement in the strong form that both the graph and range of~$\Sigma_n$-definable functions on sets exists (cf. \cite{Simpson1978aa} p. 368, \cite{Sacks1990} p. 174).

Several of the classical results about Kripke-Platek set theory easily generalize to~$\Sigma_n$-admissibles. In particular, if~$L_{\alpha}$ is~$\Sigma_n$-admissible, then (i)~$L_{\alpha}$ satisfies~$\Delta_n$-separation, (ii)~$L_{\alpha}$ models that the~$\Sigma_n$- and~$\Pi_n$-formulas are uniformly closed under bounded quantification, and (iii)~$L_{\alpha}$ satisfies~$\Sigma_n$-transfinite recursion. For the proofs of these results for the case~$n=1$, see Chapters~I-II of Devlin's book~\cite{Devlin1984aa}; the proofs for the results~$n>1$ carry over word-for-word. An idea closely related to transfinite recursion is Mostowksi Collapse. Since admissible~$L_{\alpha}$ don't necessarily model the Mostowksi Collapse Lemma, it is natural to formulate axioms pertaining directly to the Mostowski Collapse Lemma. In particular, we define:
\begin{defn} $\emph{Axiom\mbox{\;}Beta}$ says that for all sets~$X,R$ such that~$(X,R)$ is well-founded, there is a set~$\pi$ such that~$\pi$ is a function with domain~$X$ satisfying, for each~$y$ from~$X$, the equation~$\pi(y)=\{\pi(y^{\prime}): y^{\prime}\in X \; \& \; y^{\prime} R y\}$ (cf. Barwise~\cite{Barwise1975ab} Definition I.9.5 p. 39). \label{defna:axiombeta}
 \end{defn}
\noindent The set-version of the Mostowksi Collapse Lemma holds in admissible~$L_{\alpha}$ which satisfy Axiom~Beta. The set-version of this lemma states that for all sets~$X,E$ such that~$(X,E)$ is well-founded and extensional, there is a transitive set~$M$ and an isomorphism~$\pi:(X,E)\rightarrow (M,\in)$ (cf. \cite{Kunen1980} pp. 105-106, \cite{Kunen2011aa}, \cite{Kunen2011aa} p. 56~ff, \cite{Jech2003} p. 69). The traditional~${\tt ZFC}$-proof of~$\mathrm{Axiom\mbox{\;}Beta}$ uses~$\Sigma_1$-replacement and~$\Sigma_1$-separation, and so~$L_{\alpha}$ models Axiom~Beta for all~$\Sigma_2$-admissible~$\alpha$.

Other basic properties of the structures~$L_{\alpha}$ relate to its canonical well-ordering~$<_L$. The well-order~$<_L$ may be taken to be given by a canonical formula that is uniformly~$\Delta_1$ in admissible~$L_{\alpha}$. Further, since~$<_L$ is uniformly~$\Delta_1$, we have that this well-order is absolute between various admissible~$L_{\alpha}$. Moreover, one has that the function~$x\mapsto \mathrm{pred}_{<_L}(x)$ is uniformly~$\Delta_1$ in admissibles where we define~$y\in \mathrm{pred}_{<_L}(x)$ iff~$y<_L x$ (cf. Devlin~\cite{Devlin1984aa} pp. 74-75). Finally, just as the~$\Sigma_m$- and~$\Pi_m$-formulas are closed under bounded quantification for~$0\leq m\leq n$ in~$\Sigma_n$-admissibles, so for~$0< m\leq n$ they are closed under~$<_L$-bounding in~$\Sigma_n$-admissibles.

Other important properties of~$\Sigma_n$-admissibles that we shall use are related to uniformization. A structure~$M$ satisfies  {\it~$\utilde{ \Sigma}_n$-uniformization} if for every~$\utilde{\Sigma}_n^{M}$-definable relation~$R\subseteq M\times M$ there is a~$\utilde{\Sigma}_n^{M}$-definable relation~$R^{\prime}\subseteq R$ such that~$M \models \forall \; x \; [(\exists \; y \; R(x,y))\rightarrow (\exists !\;  y \; R^{\prime}(x,y))]$. In this case,~$R^{\prime}$ is called a~$\utilde{\Sigma}_n^{M}$-definable \emph{uniformization} of~$R$. In his famous paper, Jensen showed that admissible~$L_{\alpha}$ are models of~$\utilde{\Sigma}_n$-uniformization \emph{for all~$n\geq 1$} (cf.~\cite{Jensen1972aa} Theorem 3.1 p. 256 and Lemma 2.15 p. 255;~\cite{Devlin1984aa} Theorem 4.5 p. 269). The proof of this theorem is very difficult, and in fact holds for all members~$J_{\alpha}$ of Jensen's alternate hierarchy. However, in what follows we can avoid direct appeal to Jensen's Theorem by appealing to the following weak version, whose elementary proof proceeds by choosing~$<_L$-least witnesses:

\begin{prop}\label{prop:weakuniform} (\emph{Weak Uniformization}) Suppose~$n\geq 1$. If~$L_{\alpha}$ is~$\Sigma_n$-admissible then~$L_{\alpha}$ satisfies~$\utilde{\Sigma}_m$-uniformization for every~$1\leq m\leq n$. Moreover, the parameters in the~$\Sigma_m$-definition of the uniformization~$R^{\prime}$ can be taken to be the same as the parameters in the~$\Sigma_m$-definition of~$R$.
\end{prop}

\noindent Let's finally state a simple consequence of uniformization that we shall appeal to repeatedly in what follows:
\begin{prop}\label{invert}
(Proposition on Right-Inverting a Surjection) Suppose that~$n\geq 1$ and that~$L_{\alpha}$ is~$\Sigma_n$-admissible. Suppose that~$Y$ is a~$\utilde{\Sigma}_n^{L_{\alpha}}$-definable subset of~$L_{\alpha}$ and~$X$ is a subset of~$L_{\alpha}$. Suppose there is a~$\utilde{\Sigma}_n^{L_{\alpha}}$-definable surjection~$\pi:Y\rightarrow X$. Then~$X$ is a~$\utilde{\Sigma}_n^{L_{\alpha}}$-definable subset of~$L_{\alpha}$ and there is a~$\utilde{\Sigma}_n^{L_{\alpha}}$-definable injection~$\iota:X\rightarrow Y$ satisfying~$\pi\circ \iota  = \mathrm{id}_X$.
\end{prop}

An important concept in what follows is the~$n$-th projectum of the structure~$L_{\alpha}$. This was introduced by Kripke (\cite{Kripke1964aa}) and Platek (\cite{Platek1966aa}), and it  records how small one can possibly make~$\alpha$ under a~$\Sigma_n$-definable injection:
\begin{defn}\label{defn:project}
Suppose that~$n> 0$ and~$\alpha>\omega$. Then the {\it~$n$-th projectum}~$\rho_n(\alpha)=\rho_n$ of~$\alpha$ is the least~$\rho\leq \alpha$ such that there is a~$\utilde{\Sigma}_n^{L_{\alpha}}$-definable injection~$\iota:\alpha\rightarrow \rho$.
\end{defn}
\noindent There are several different equivalent characterizations of the~$n$-th projectum (cf. \cite{Sacks1990} p. 157, \cite{Barwise1975ab} Definition V.6.1 p. 174, \cite{Jensen1972aa} pp. 256-257, \cite{Schindler2010aa} Definition 2.1 p. 619). In particular, for admissible~$\alpha$, the~$n$-th projectum may be equivalently defined as the  smallest~$\rho\leq \alpha$ such that there is a~$\utilde{\Sigma}_n^{L_{\alpha}}$-definable injection~$\iota:L_{\alpha}\rightarrow \rho$. 

Another basic tool that we employ is the notion of a~$\Sigma_n$-elementary substructure. Recall that if~$M$ and~$N$ are structures in the signature of~${\tt ZFC}$, then~$M\prec_n N$ is said to hold, and~$M$ is said to be a~$\Sigma_n$-\emph{elementary substructure} of~$N$, if~$M\subseteq N$ and for every~$\Sigma_n$-formula~$\varphi(\overline{x})$ and every tuple of parameters~$\overline{a}$ from~$M$, it is the case that~$M\models \varphi(\overline{a})$ if and only if~$N\models \varphi(\overline{a})$.  Here are some basic facts about~$\Sigma_n$-elementary substructures and the constructible hierarchy that we shall use:
\begin{prop}\label{prop:whatami234123}
\begin{enumerate}
\item[] 
\item \emph{The~$\Sigma_n$-Definable Closure is a~$\Sigma_n$-Elementary Substructure}: Suppose that~$L_{\alpha}$ is~$\Sigma_n$-admissible and~$A\subseteq L_{\alpha}$. Let \emph{the~$\Sigma_n$-definable closure of~$L_{\alpha}$ with parameters~$A$}, written~$\mathrm{dcl}^{L_{\alpha}}_{\Sigma_n}(A)$, denote the set of elements~$a$ of~$L_{\alpha}$ such that there is a~$\Sigma_n$-formula~$\varphi(x,\overline{y})$ with all free variables displayed and parameters~$\overline{p}\in A$ such that~$L_{\alpha}\models \varphi(a,\overline{p}) \wedge \forall \; a^{\prime} \; (\varphi(a^{\prime},\overline{p})\rightarrow a=a^{\prime})$. Then~$\mathrm{dcl}^{L_{\alpha}}_{\Sigma_n}(A)\prec_n L_{\alpha}$. \label{eqn:denclos}

\item If~$\kappa$ is an uncountable regular cardinal, then~$L_{\kappa}$ is a model of~${\tt ZFC}\mbox{-}{\tt P}$ (cf. Definition~\ref{defn:ZFCminusP}). \label{eqn:helper1}

\item \emph{Admissibility and Axiom Beta Preserved Under Elementary Substructure}: Suppose that~$n\geq 1$ and that~$L_{\alpha}\prec_n L_{\beta}$ where~$\beta$ is~$\Sigma_n$-admissible. Then~$\alpha$ is~$\Sigma_n$-admissible. Further, if~$L_{\beta}\models \mathrm{Axiom\mbox{\;}Beta}$ then~$L_{\alpha}\models \mathrm{Axiom\mbox{\;}Beta}$. \label{simpelthings}

\item \emph{Consequence of~${\tt V\hspace{-1mm}=\hspace{-1mm}L}$ for~$\Sigma_1$-Substructures of~$L$ up to a Successor Cardinal}: Suppose that~${\tt V\hspace{-1mm}=\hspace{-1mm}L}$ and~$\lambda$ is an infinite cardinal and~$\lambda\cup \{\lambda\}\subseteq M$,~$M\prec_1 L_{\lambda^{+}}$,~$\left|M\right|=\lambda$. Then~$M=L_{\gamma}$ for some~$\gamma$ with~$\left|\gamma\right|= \lambda$.\label{eqn:genpop}
\end{enumerate}
\end{prop}
\begin{proof}
For the first item, see  Devlin~\cite{Devlin1984aa} Lemma II.5.3 p. 83 which proves the result for~$n=\omega$; the same proof works for~$1\leq n<\omega$. For the next item on ${\tt ZFC}\mbox{-}{\tt P}$, see ~\cite{Kunen1980} p. 177,~\cite{Kunen2011aa} Lemma II.6.22 p. 139, ~\cite{Jech2003} p. 198, noting that these same proofs also give one the collection schema from the official definition of  ${\tt ZFC}\mbox{-}{\tt P}$ (cf. Definition~\ref{defn:ZFCminusP}). The proof of the third item follows from a routine induction. For the fourth and final item, see Devlin~\cite{Devlin1984aa} Lemma II.5.10 p. 85 for the special case~$\lambda =\omega$; the same proof works for the general case.
\end{proof}

\section{Construction and Existence Theorems, and Joint Consistency Problem}\label{sec04.5}

The aim of this section is to build models of abstraction principles with predicative amounts of comprehension, and these yield our solution to the joint consistency problem described at the close of  \S\ref{sec0.1555predab}. The first step is the following construction. This construction is also an important part of the proof of our Main Theorem~\ref{thm:main}, whose proof is presented in the next section \S\ref{sec05}. In the statement of this construction theorem, the key concepts of~$\Sigma_n$-admissible and~$n$-th projectum~$\rho_n$ were defined in the previous section \S\ref{sec02.2}. Likewise, recall that the theories~${\tt SO}$ and~${\tt \Sigma^1_1\mbox{-}[E_1, \ldots, E_k]A}$ were defined respectively in Definition~\ref{eqn:secondorderlogic} and Definition~\ref{predabstracdefn} from \S\ref{sec0.1555predab}.

\begin{thm}\label{thm:BCT} (\emph{Construction Theorem}). Suppose that~$n\geq 1$ and that~$\alpha$ is~$\Sigma_n$-admissible with~$\rho_n(\alpha)=\rho<\alpha$ and let~$\iota:L_{\alpha}\rightarrow \rho$ be a witnessing~${\utilde \Sigma}_n^{L_{\alpha}}$-definable injection. Then consider the following structure~$\mathcal{M}$ in the signature of~${\tt SO}$~(cf. Definition~\ref{eqn:secondorderlogic}):
\begin{equation}\label{eqn:555666}
\mathcal{M} = (\rho, P(\rho)\cap L_{\alpha}, P(\rho\times \rho)\cap L_{\alpha}, \ldots)
\end{equation}
Further, suppose for each~$i\in [1,k]$, the relation~$E_i$ is a~${\utilde \Sigma}^{1, \mathcal{M}}_{n\mbox{-}1}$-definable equivalence relation on~$(P(\rho^{m_i})\cap L_{\alpha})$. 

Then consider the~${\utilde \Sigma}_{n}^{L_{\alpha}}$-definable maps~$\partial_{E_i}:(P(\rho^{m_i})\cap L_{\alpha})\rightarrow \rho$ defined by~$\partial_{E_i}(X)=\iota(\ell_i(X))$ where~$\ell_i(X)$ is the~$<_L$-least member of~$X$'s~$E_i$-equivalence class. Then the following expansion of~$\mathcal{M}$ is a model of the theory~${\tt \Sigma^1_1\mbox{-}[E_1, \ldots, E_k]A}+{\tt GC}$, where the global well-order~$<$ on objects is given by the membership relation on the ordinal~$\rho$:
\begin{equation}\label{eqn:thestructure}
\mathcal{N}=(\rho, P(\rho)\cap L_{\alpha}, P(\rho\times \rho)\cap L_{\alpha}, \ldots, \partial_{E_1}, \ldots, \partial_{E_k},<)
\end{equation}
\end{thm}
\begin{proof}
For each~$i\in [1,k]$, define
\begin{equation}
\widehat{E}_i = \{(X,Y)\in (P(\rho^{m_i})\cap L_{\alpha})\times (P(\rho^{m_i})\cap L_{\alpha}): \mathcal{M}\models E_i(X,Y)\}
\end{equation}
Then since~$E_i$ is~${\utilde \Sigma}^{1, \mathcal{M}}_{n\mbox{-}1}$-definable, it follows that~$\widehat{E}_i$ is~${\utilde \Sigma}^{L_{\alpha}}_{n\mbox{-}1}$-definable, so that~$\widehat{E}_i$ is a~${\utilde \Sigma}^{L_{\alpha}}_{n\mbox{-}1}$-definable equivalence relation on~$(P(\rho^{m_i})\cap L_{\alpha})$. For each element~$X$ of~$(P(\rho^{m_i})\cap L_{\alpha})$, let~$[X]_{\widehat{E}_i}\subseteq (P(\rho^{m_i})\cap L_{\alpha})$ denote the~$\widehat{E}_i$-equivalence class of~$X$. Then~$\ell_i:(P(\rho^{m_i})\cap L_{\alpha})\rightarrow (P(\rho^{m_i})\cap L_{\alpha})$ is defined by~$\ell_i(X)=\min_{<_L}([X]_{\widehat{E}_i})$, and its graph has the following definition:
\begin{align}
\ell_i(X)=Y \Longleftrightarrow X,Y\in L_{\alpha} & \;\&\; \; X,Y\subseteq \rho^{m_i} \; \& \; \widehat{E}_i(X,Y)\label{defnLekki}\\
&   \; \& \; \forall Z<_L Y \; [Z\subseteq \rho^{m_i}\rightarrow \neg \widehat{E}_i(X,Z)] \notag
\end{align}
Since adding quantifiers bounded by~$<_L$ to~$\Sigma_m$- or~$\Pi_m$-formulas for~$m\leq n$ does not increase their complexity, we have that the graph of~$\ell_i$ is defined by the conjunction of a~${\utilde \Sigma}_{n\mbox{-}1}^{L_{\alpha}}$-formula with a~${\utilde \Pi}_{n\mbox{-}1}^{L_{\alpha}}$-formula and so is~${\utilde \Sigma}_{n}^{L_{\alpha}}$-definable. Then the map~$\partial_{E_i}: (P(\rho^{m_i})\cap L_{\alpha})\rightarrow \rho$ is defined by~$\partial_{E_i}(X) = \iota(\ell_i(X))$, which is likewise~${\utilde \Sigma}_{n}^{L_{\alpha}}$-definable since it is the composition of two~${\utilde \Sigma}_{n}^{L_{\alpha}}$-definable functions. (Note that in the case~$n=1$, the function~$\ell_i$ is defined by the conjunction of a~${\utilde \Sigma}_{0}^{L_{\alpha}}$-formula with a~${\utilde \Sigma}_{1}^{L_{\alpha}}$-formula and so is~${\utilde \Sigma}_{1}^{L_{\alpha}}$-definable. For, the formula~$\forall \; Z<_L Y \; \theta(Z,Y)$ for any~${\utilde \Sigma}_{0}^{L_{\alpha}}$-definable~$\theta(Z,Y)$ is equivalent to the formula $\exists \; Y^{\prime}  \; Y^{\prime}=\mathrm{pred}_{<_L}(Y) \; \& \; \forall \; Z\in Y^{\prime} \; \theta(Z,Y)$ which is~${\utilde \Sigma}_{1}^{L_{\alpha}}$-definable because the map~$Y\mapsto \mathrm{pred}_{<_L}(Y)$ is~${\utilde \Delta}_{1}^{L_{\alpha}}$-definable). 

Now let us argue that the so-defined structure~$\mathcal{N}$ from equation~(\ref{eqn:thestructure}) satisfies the abstraction principle~$A[E_i]$~(\ref{defnAe}). First suppose that~$\mathcal{N}\models \partial_{E_i}(X)=\partial_{E_i}(Y)$ for some~$X,Y\in (P(\rho^{m_i})\cap L_{\alpha})$. Then since~$\iota:L_{\alpha}\rightarrow \rho$ is an injection we have that~$\ell_i(X)=\ell_i(Y)$, so that~$\min_{<_L}([X]_{\widehat{E}_i})=\min_{<_L}([Y]_{\widehat{E}_i})$. Hence~$\widehat{E}_i(X,Y)$ so that~$\mathcal{M}\models E_i(X,Y)$ and hence its expansion~$\mathcal{N}$ also models this. Conversely, suppose that~$\mathcal{N}\models E_i(X,Y)$, so that its reduct~$\mathcal{M}$ also models this. Then~$\widehat{E}_i(X,Y)$ and hence~$[X]_{\widehat{E}_i}= [Y]_{\widehat{E}_i}$ and~$\min_{<_L}([X]_{\widehat{E}_i}) = \min_{<_L}([Y]_{\widehat{E}_i})$, so that~$\ell_i(X)=\ell_i(Y)$ and hence~$\partial_{E_i}(X)=\partial_{E_i}(Y)$. Hence in fact the structure~$\mathcal{N}$ from equation~(\ref{eqn:thestructure}) satisfies the abstraction principle~$A[E_i]$.

So now it remains to show that the structure~$\mathcal{N}$ from equation~(\ref{eqn:thestructure}) satisfies the First-Order Comprehension Schema (cf. Definition~\ref{pred:comp:schema}) and the~$\Sigma^1_1$-Choice Schema (cf. Definition~\ref{sigam11choice}) in the expanded signature containing the function symbols~$\partial_{E_1}, \ldots, \partial_{E_k}$. First let us establish a preliminary result that any $\utilde{\Sigma}^{1,\mathcal{N}}_1$-definable subset of $\mathcal{N}$ is $\utilde{\Sigma}^{L_{\alpha}}_n$-definable (indeed, in the same parameters). This result is proven by induction on the complexity of the formula defining the subset of $\mathcal{N}$. By a subset of the many-sorted structure~$\mathcal{N}$, we mean any subset of any finite product $S_1\times \cdots \times S_n$, wherein $S_i$ is one of the sorts $\rho, P(\rho)\cap L_{\alpha}, P(\rho\times \rho)\cap L_{\alpha}, \ldots$ of the structure~$\mathcal{N}$ as displayed in equation~(\ref{eqn:thestructure}). So our preliminary result establishes not only that $\utilde{\Sigma}^{1,\mathcal{N}}_1$-definable subsets of the first-order part $\rho$ are $\utilde{\Sigma}^{L_{\alpha}}_n$-definable, but also that e.g. $\utilde{\Sigma}^{1,\mathcal{N}}_1$-definable subsets of $\rho\times (P(\rho)\cap L_{\alpha})$ are $\utilde{\Sigma}^{L_{\alpha}}_n$-definable.

As a base case, we show that any subset of $\mathcal{N}$ defined by an atomic formula is $\utilde{\Delta}^{L_{\alpha}}_n$-definable. Recall that an \emph{unnested} atomic formula in a signature is one of the form $x=y$, $c=y$, $f(\overline{x})=y$ or $P\overline{x}$, where $c, f, P$ are respectively constant symbols, function symbols, and relation symbols of the signature (cf. \cite{Hodges1993aa} p. 58). Then in $\mathcal{N}$, any atomic formula $\varphi$ is equivalent to both a $\utilde{\Sigma}^{1, \mathcal{N}}_1$-formula~$\varphi^{\exists}\equiv \exists \; \overline{R} \; \exists \; \overline{y} \; \varphi^{\exists}_0$ and a $\utilde{\Pi}^{1, \mathcal{N}}_1$-formula $\varphi^{\forall}\equiv \forall \; \overline{R} \; \forall \; \overline{y} \; \varphi^{\forall}_0$ in which $\varphi^{\exists}_0$ and $\varphi^{\forall}_0$ are quantifer-free and in which any atomic subformula of them is unnested (cf. \cite{Hodges1993aa} Theorem 2.6.1 p. 58). Now, the unnested atomics of the signature of $\mathcal{N}$ are $\partial_{E_i}(R)=x$, $R[\overline{x}]=S$, $R\overline{y}$, and $x=y$. The first is by construction $\utilde{\Sigma}^{L_{\alpha}}_n$-definable and the last three are trivially $\utilde{\Sigma}^{L_{\alpha}}_n$-definable. Further, their negations are likewise $\utilde{\Sigma}^{L_{\alpha}}_n$-definable: for instance, $\partial_{E_i}(R)\neq x$ iff $\exists \; y\in \rho \; (\partial_{E_i}(R)=y \; \& \; y\neq x)$, and $\Sigma_n$-formulas are closed under bounded quantification in $L_{\alpha}$. Now, without loss of generality, we may assume that the negations in $\varphi^{\exists}_0$ and $\varphi^{\forall}_0$ are all pushed to the inside, so that they apply only to unnested atomics. Then since the $\Sigma_n$-formulas are closed under finite union and intersection in $L_{\alpha}$, we have that each subformula of $\varphi^{\exists}_0$ and $\varphi^{\forall}_0$ is $\utilde{\Sigma}^{L_{\alpha}}_n$-definable and thus so are they themselves. By the same reasoning, this holds true for their negations  $\neg \varphi^{\exists}_0$ and $\neg \varphi^{\forall}_0$ as well. Since $\varphi^{\exists}$ is formed from $\varphi^{\exists}_0$ by adding a single block of existential quantifiers, we have that $\varphi^{\exists}$ is $\utilde{\Sigma}^{L_{\alpha}}_n$-definable. Since $\neg (\varphi^{\forall})$ is formed from $\neg (\varphi^{\forall}_0)$ by adding a single block of existential quantifiers, we have that $\neg (\varphi^{\forall})$ is $\utilde{\Sigma}^{L_{\alpha}}_n$-definable, so that  $ \varphi^{\forall}$ is $\utilde{\Pi}^{L_{\alpha}}_n$-definable. Since the original atomic formula~$\varphi$ is equivalent to both  $\varphi^{\exists}$ and $\varphi^{\forall}$, we have that the original atomic formula~$\varphi$ is indeed $\utilde{\Delta}^{L_{\alpha}}_n$-definable.

Then we show that any $\utilde{\Sigma}^{1,\mathcal{N}}_1$-definable subset of $\mathcal{N}$ is $\utilde{\Sigma}^{L_{\alpha}}_n$-definable by a straightforward induction. We may again assume that all the negations are pushed to the inside and apply only to atomics. And the atomics and negated atomics are  $\utilde{\Sigma}^{L_{\alpha}}_n$-definable by the previous paragraph. Since  $\utilde{\Sigma}^{L_{\alpha}}_n$-definability is closed under finite union and intersection, the induction steps for conjunction and disjunction are trivial. Likewise, the induction steps for first-order quantification hold because first-order quantification in $\mathcal{N}$ corresponds to bounded quantification over elements of $\rho$ in $L_{\alpha}$, and the $\Sigma_n$-formulas are closed under bounded quantification in $L_{\alpha}$. So the inductive argument up this point establishes the result for formulas with no higher-order quantifiers. But since we're restricting to the case of $\Sigma^{1,\mathcal{N}}_1$-definable subsets, the addition of a block of higher-order existential quantifiers ranging over $P(\rho^k)\cap L(\alpha)$ does not bring us out of the complexity class~$\utilde{\Sigma}^{L_{\alpha}}_n$.

Hence indeed any $\utilde{\Sigma}^{1,\mathcal{N}}_1$-definable subset of $\mathcal{N}$ is $\utilde{\Sigma}^{L_{\alpha}}_n$-definable. From this, the First-Order Comprehension Schema (cf. Definition~\ref{pred:comp:schema}) follows directly from $\Delta_n$-separation in $L_{\alpha}$. As for the $\Sigma^1_1$-choice schema, suppose that~$\mathcal{N} \models \forall \; x \; \exists \; R \; \varphi(x,R)$, wherein~$\varphi$ is~$\Sigma^1_1$. Choose a $\utilde{\Sigma}_n$-formula $\psi$ such that for all $x\in \rho$ and $R\in (P(\rho)\cap L_{\alpha})$ one has $\mathcal{N}\models \varphi(x,R)$ iff $L_{\alpha}\models \psi(x,R)$. Then one has that $L_{\alpha}\models \forall x\in \rho \; \exists \; R\subseteq \rho \; \psi(x,R)$. Define $\Gamma(x,R)\equiv [x\in \rho \; \& \; R\in (P(\rho)\cap L_{\alpha}) \; \& \; L_{\alpha}\models \psi(x,R)]$. By weak uniformization~(\ref{prop:weakuniform}), choose a $\Sigma^{L_{\alpha}}_n$-definable uniformization $\Gamma^{\prime}$ of $\Gamma$. Since $\Gamma^{\prime}:\rho\rightarrow (P(\rho)\cap L_{\alpha})$, it follows from~$\Sigma_n$-replacement that its graph is an element of~$L_{\alpha}$, and obviously we have  $L_{\alpha}\models \forall \; x\in \rho \; \psi(x,\Gamma^{\prime}(x))$. Then define $R^{\prime}xy$ if and only if $y\in \Gamma^{\prime}(x)$, so that $R^{\prime}\in (P(\rho\times \rho)\cap L_{\alpha})$ and $R^{\prime}[x]=\Gamma(x)$. Then one has  $L_{\alpha}\models \forall \;x\in \rho \; \psi(x,R^{\prime}[x])$ and hence $\mathcal{N}\models \varphi(x,R^{\prime}[x])$.

As for the global choice principle~${\tt GC}$, we may briefly note that~$\mathcal{N}$ obviously satisfies it when we use the ordinary ordering~$<$ on the ordinal~$\rho$ as the witness. For, since the ordering~$<$ on~$\rho$ is~$\Delta_0$-definable, it exists in~$P(\rho\times \rho)\cap L_{\alpha}$ by~$\Delta_0$-separation on the set~$\rho\times \rho$ in~$L_{\alpha}$. In the previous paragraphs, we have verified that various forms of comprehension hold on~$\mathcal{N}$, in which parameters are allowed to occur. Hence these forms of comprehension continue to hold when~$<$ is permitted to occur within the formulas because we can view this as simply yet another parameter.
\end{proof}

\begin{thm}\label{thm:basicexistence} (\emph{Existence Theorem}).
Let~$\gamma\geq 0$ and let~$\lambda = \omega_{\gamma}^L$ and~$\kappa=\omega_{\gamma+1}^L$. Then for each~$n\geq 1$ there is an~$\Sigma_n$-admissible~$\alpha_n$ such that
\begin{equation}
\lambda<\alpha_n<\kappa, \hspace{5mm} \rho_n(\alpha_n)=\lambda, \hspace{5mm} L_{\alpha_n}\prec_n L_{\kappa}, \hspace{5mm} L_{\alpha_n}\models \mathrm{Axiom\;Beta}
\end{equation}
More specifically, we can choose~$\alpha_n$ so that~$L_{\alpha_n} =\mathrm{dcl}_{\Sigma_n}^{L_{\kappa}}(\lambda \cup \{\lambda\})$. Further, the following set~$\mathcal{F}_n\subseteq \lambda$ is~$\Sigma_1^{L_{\alpha_n}}$-definable, wherein~$\langle \cdot, \cdot\rangle:\lambda \times \lambda \rightarrow \lambda$ is G\"odel's~$\Sigma_1$-definable pairing function and~$\mathrm{Form}(\Sigma_n)$ is the set of G\"odel numbers of~$\Sigma_n$-formulas:
\begin{equation}\label{eqn:defnbigO}
\mathcal{F}_n = \{\langle \ulcorner \varphi(x, \overline{y},z)\urcorner, \overline{\beta}\rangle: \varphi(x,\overline{y},z)\in \mathrm{Form}(\Sigma_n) \; \& \; \overline{\beta} <\lambda\}
\end{equation}
Moreover, there is a~$\utilde{\Sigma}^{L_{\alpha_n}}_n$ surjective partial map~$\theta_n:\mathcal{F}_n \dashrightarrow L_{\alpha_n}$ such that
\begin{align}
 \theta_n(\langle \ulcorner \varphi(x, \overline{y},z)\urcorner, \overline{\beta}\rangle) =a  \Longrightarrow &   L_{\alpha_n} \models \varphi(a,\overline{\beta}, \lambda) \label{eqn:idoall} \\
(L_{\alpha_n} \models \exists \; ! \; x\; \varphi(x,\overline{\beta}, \lambda))& \Longrightarrow    \langle \ulcorner \varphi(x, \overline{y},z)\urcorner, \overline{\beta}\rangle\in \mathrm{dom}(\theta_n) \nonumber
\end{align}
and a~$\utilde{\Sigma}^{L_{\alpha_n}}_n$-definable injection~$\iota_n: L_{\alpha_n}\rightarrow \mathrm{dom}(\theta_n)$ such that~$\theta_n \circ \iota_n$ is the identity on~$L_{\alpha_n}$. Further, the sequence~$\alpha_n$ is strictly increasing. Finally, for each~$n\geq 1$ there is an injection~$\chi_n:\lambda\rightarrow \theta^{-1}_n(\{0,1\})$ whose graph is in~$L_{\alpha_n}$.
\end{thm}
\begin{proof}
Since the result is absolute, we may assume~${\tt V\hspace{-1mm}=\hspace{-1mm}L}$, and hence we may assume that~$\lambda$ and~$\kappa$ are cardinals. Since~$\kappa$ is regular uncountable one has that~$L_{\kappa}\models {\tt ZFC\mbox{-}P}$ (cf. Proposition~\ref{prop:whatami234123}, item \ref{eqn:helper1}). Let~$M=\mathrm{dcl}_{\Sigma_n}^{L_{\kappa}}(\lambda \cup \{\lambda\})$. Since the~$\Sigma_n$-definable closure is a~$\Sigma_n$-elementary substructure~(cf. Proposition~\ref{prop:whatami234123}, item~\ref{eqn:denclos}), we have~$M\prec_n L_{\kappa}$. Since~$\kappa=\lambda^{+}$ and~$\lambda\cup \{\lambda\}\subseteq M$ and~$M\prec_1 L_{\lambda^{+}}$, it follows from the consequence of~${\tt V\hspace{-1mm}=\hspace{-1mm}L}$ for~$\Sigma_1$-substructures of~$L$ up to a successor cardinal (Proposition~\ref{prop:whatami234123} item \ref{eqn:genpop}) that~$M=L_{\alpha_n}$ where~$\left|\alpha_n\right|= \lambda$. Then~$\lambda \leq \alpha_n<\kappa$. But since~$\lambda\in M=L_{\alpha_n}$ we have~$\lambda <\alpha_n<\kappa$. By Proposition~\ref{prop:whatami234123} item \ref{simpelthings}, we also have that~$L_{\alpha_n}$ is~$\Sigma_n$-admissible and satisfies Axiom~Beta.

Then define the following relation~$R_n\subseteq \mathcal{F}_n\times L_{\alpha_n}$ by
\begin{equation}
R_n(\langle \ulcorner \varphi(x, \overline{y},z)\urcorner, \overline{\beta}\rangle, a)\Longleftrightarrow  \langle \ulcorner \varphi(x, \overline{y},z)\urcorner, \overline{\beta}\rangle\in \mathcal{F}_n  \; \& \;\;  L_{\alpha_n}\models \varphi(a, \overline{\beta}, \lambda)
\end{equation}
Then by the definability of partial satisfaction predicates,~$R_n$ is~$\utilde{\Sigma}^{L_{\alpha_n}}_n$-definable. Then by weak uniformization (Proposition~\ref{prop:weakuniform}), choose a~$\utilde{\Sigma}^{L_{\alpha_n}}_n$-definable uniformization~$\theta_n:\mathcal{F}_n\dashrightarrow L_{\alpha_n}$ of~$R_n$. Then~$\theta_n$ is a surjective partial function. For, suppose that~$a\in L_{\alpha_n}$. Since~$L_{\alpha_n}=\mathrm{dcl}_{\Sigma_n}^{L_{\kappa}}(\lambda \cup \{\lambda\})$, there is~$\Sigma_n$-formula~$\varphi(x, \overline{y},z)$ and~$\overline{\beta}< \lambda$ such that 
\begin{equation}\label{eqn:whatigot4}
L_{\alpha_n} \models \varphi(a,\overline{\beta},\lambda) \; \& \; [\forall \; x \; (\varphi(x, \overline{\beta}, \lambda)\rightarrow x=a)]
\end{equation}
Then~$L_{\alpha_n} \models \exists \; x \; \varphi(x,\overline{\beta}, \lambda)$ and~$ \langle \ulcorner \varphi(x, \overline{y},z)\urcorner, \overline{\beta}\rangle$ is in~$\mathcal{F}_n$. Then on the input~$u= \langle  \ulcorner \varphi(x, \overline{y},z)\urcorner, \overline{\beta}\rangle$, we have that~$\theta_n(u)$ is defined and if~$\theta_n(u)=a^{\prime}$ then~$L_{\alpha_n} \models \varphi(a^{\prime},\overline{\beta},\lambda)$. But in conjunction with equation~(\ref{eqn:whatigot4}), it thus follows that~$a=a^{\prime}=\theta_n(u)$. Hence, indeed~$\theta_n:\mathcal{F}_n\dashrightarrow L_{\alpha_n}$ is a surjective partial function. Let~$\mathcal{F}^{\prime}_n\subseteq \mathcal{F}_n$ be the domain of~$\theta_n$, which is likewise~$\utilde{\Sigma}^{L_{\alpha_n}}_n$-definable. By the Proposition on Right-Inverting a Surjection (Proposition~\ref{invert}), it follows that there is a~$\utilde{\Sigma}^{L_{\alpha_n}}_n$-definable injection~$\iota_n: L_{\alpha_n}\rightarrow \mathcal{F}^{\prime}_n$ such that~$\theta_n \circ \iota_n=\mathrm{id}_{L_{\alpha_n}}$. Since~$\mathcal{F}_n\subseteq \lambda$, we then have that~$\rho_n(\alpha_n)\leq \lambda$. Since~$\lambda$ is a cardinal and~$L_{\alpha_n}$ has cardinality~$\lambda$, we must have then that~$\rho_n(\alpha_n)= \lambda$.

Now we argue that~$\alpha_1<\alpha_2<\alpha_3<\cdots$. Since~$L_{\alpha_n}=\mathrm{dcl}_{\Sigma_n}^{L_{\kappa}}(\lambda \cup \{\lambda\})$, we have that~$L_{\alpha_n}\subseteq L_{\alpha_{n+1}}$ and hence that~$\alpha_n\leq \alpha_{n+1}$. Suppose that it was not always that case that~$\alpha_n<\alpha_{n+1}$ for all~$n\geq 1$. Then~$\alpha_n=\alpha_{n+1}$ for some~$n\geq 1$. Since~$L_{\alpha_{n+1}}$ is~$\Sigma_{n\mbox{+}1}$-admissible and~$L_{\alpha_n}=L_{\alpha_{n+1}}$, we have that~$L_{\alpha_{n}}$ is~$\Sigma_{n\mbox{+}1}$-admissible and so satisfies~$\Sigma_{n}$-separation. Hence~$\mathcal{F}_n^{\prime}\in L_{\alpha_n}$ and hence by~$\Sigma_n$-replacement, the~$\utilde{\Sigma}_n^{L_{\alpha_n}}$-definable map~$\theta_n:\mathcal{F}_n^{\prime}\rightarrow L_{\alpha_n}$ would be bounded and thus not surjective.

Finally, we verify that for each~$n\geq 1$ there is injection~$\chi_n:\lambda\rightarrow \theta^{-1}_n(\{0,1\})$. Let~$\varphi(x,y,z)$ say ``$x=0$ and~$y$ is an ordinal.'' Then for each~$\beta<\lambda$ there is exactly one~$x$ in~$L_{\alpha_n}$ such that~$L_{\alpha_n}\models \varphi(x, \beta, \lambda)$. Then by equation~(\ref{eqn:idoall}) we have that~$\langle \ulcorner \varphi(x,y,z)\urcorner, \beta\rangle\in \mathrm{dom}(\theta_n)$ and we have by equation~(\ref{eqn:idoall}) that {}$\theta_n(\langle \ulcorner \varphi(x,y,z)\urcorner, \beta\rangle)=0$. Then define the function~$\chi_n:\lambda\rightarrow \theta^{-1}_n(\{0,1\})$ by~$\chi_n(\beta)=\langle \ulcorner \varphi(x,y,z)\urcorner, \beta\rangle$, which is clearly injective; further clearly the graph of~$\chi$ is in~$L_{\alpha_n}$.
\end{proof}

The extra information about the injection~$\chi_n:\lambda\rightarrow \theta^{-1}_n(\{0,1\})$ in Theorem~\ref{thm:basicexistence} will be primarily useful for our companion paper~\cite{Walsh2014ae}, where we use constructible sets to build models of an intensional type theory (cf. \S{5} of~\cite{Walsh2014ae}). The reason for the focus on $\{0,1\}$ is that in the tradition of type-theory these are used as ersatzes for the truth-values $\{F,T\}$. The map~$\chi_n$ then allows us to inject ordinals $\beta<\lambda$ into intensional entities $\chi_n(\beta)$ which determine truth-values $\theta_n(\chi_n(\beta))$, roughly after the manner in which we inject natural numbers~$e$ into algorithms~$P_e$ which determine computable number-theoretic functions~$\varphi_e$.

The proofs of these results can be seen as a generalization of our earlier constructions of models of~${\tt \Sigma^1_1\mbox{-}LB}_0$ of the form~$\mathcal{N}=(\omega, \mathrm{HYP}, \ldots, \partial)$ (\cite{Walsh2012aa} Theorem~53 p. 1695). Here~$\mathrm{HYP}$ denotes the hyperarithmetic subsets of natural numbers and~$\partial(Y)=\langle b,e\rangle$ only if~$b$ is a code for a computable ordinal~$\beta$ and~$Y$ is computable from~$b$'s canonical coding~$H_b$ of the~$\beta$-th Turing jump by the program~$e$. This earlier result can be seen as a special case of these results by virtue of the fact that if~$\alpha=\omega_1^{\mathrm{CK}}$ then~$P(\omega)\cap L_{\alpha}=\mathrm{HYP}$ (cf. Sacks~\cite{Sacks1990} \S~III.9 Exercise 9.12 p. 87). The primary difference between the proofs here and our earlier constructions of models of~${\tt \Sigma^1_1\mbox{-}LB}_0$ (\cite{Walsh2012aa} Theorem~53 p. 1695) was that the latter used Kond{\^o}'s Uniformization Theorem (\cite{Simpson2009aa} p. 224,~\cite{Kechris1995} p. 306), while the proof here used uniformization results in the constructible hierarchy like weak uniformization~\ref{prop:weakuniform}. Further, our results here can cover not just Basic Law~V, but the abstraction principles described in \S\ref{sec0.1555predab}.

Finally, we can now prove the main results on the consistency of abstraction principles in the predicative setting, which were first stated and motivated in \S\ref{sec0.1555predab}. As for Theorem~\ref{cor:consistency}, this is a limiting case of the Joint Consistency Theorem~\ref{thm:jointconsistency}. So it remains to establish this latter theorem:
\begin{proof} (\emph{of Joint Consistency Theorem}~\ref{thm:jointconsistency}):
So suppose that the formulas~$E_1(R,S)$, \ldots,~$E_k(R,S)$ in the signature of~${\tt SO}+{\tt GC}$ are provably equivalence relations on~$m_i$-ary concepts in~${\tt SO}+{\tt GC}$.

The theory~${\tt SO}$ from Definition~\ref{eqn:secondorderlogic} can be naturally written as the union of theories~${\tt SO}_m$, where ${\tt SO}_m$ restricts the instances of the Full Comprehension Schema~(\ref{eqn:unaryfullcomp}) in the signature of~${\tt SO}+{\tt GC}$ to its $\Sigma^1_m$-instances, where this is the standard notion for formulas which begin with $m$-alternating blocks of quantifiers, the first of which is a block of second-order existential quantifiers. By compactness, it suffices to show, for each $m\geq 1$, that the theory~${\tt \Sigma^1_1\mbox{-}[E_1, \ldots, E_k]A}+{\tt SO}_m+{\tt GC}$ is consistent.

Let us then fix, for the remainder of this proof, an $m\geq 1$. Choose $n>m$ sufficiently large so that (i)~the formulas~$E_1(R,S)$, \ldots,~$E_k(R,S)$ are provably equivalence relations in ${\tt SO}_{n\mbox{-}1}+{\tt GC}$, and (ii) each of the formulas~$E_1(R,S)$, \ldots,~$E_k(R,S)$ is a $\Sigma^1_{n\mbox{-}1}$-formula.

By the Existence Theorem~\ref{thm:basicexistence}, choose $\Sigma_n$-admissible $\alpha$ such that $\rho=\rho(\alpha)<\alpha$ and $L_{\alpha}$ is a model of Axiom~Beta. Then consider the following structure~$\mathcal{M}$ in the signature of~${\tt SO}$, as was also featured in equation~(\ref{eqn:555666}) of the hypothesis of the Construction Theorem~\ref{thm:BCT}:
\begin{equation}\label{eqn:555666redux}
\mathcal{M} = (\rho, P(\rho)\cap L_{\alpha}, P(\rho\times \rho)\cap L_{\alpha}, \ldots)
\end{equation}
Since $\alpha$ is $\Sigma_n$-admissible, one has that $L_{\alpha}$ satisfies $\Sigma_{n-1}$-separation. Thus the structure~$\mathcal{M}$ from equation~(\ref{eqn:555666redux}) satisfies the theory ${\tt SO}_{n-1}$ since instances of the Full-Comprehension Schema~(\ref{eqn:unaryfullcomp}) in the signature of ${\tt SO}$ associated to $\Sigma^1_{n-1}$-formulas correspond naturally to $\Sigma_{n-1}$-instances of separation in $L_{\alpha}$ on the set~$\rho$. When expanded by the well-order~$<$ on objects given by the membership relation on~$\rho$, it likewise satisfies the theory ${\tt SO}_{n-1}+{\tt GC}$.  Since the formulas~$E_1(R,S)$, \ldots,~$E_k(R,S)$ are provably equivalence relations on $m_i$-ary concepts in ${\tt SO}_{n\mbox{-}1}+{\tt GC}$, it follows that they are likewise a~${\utilde \Sigma}^{1, \mathcal{M}}_{n\mbox{-}1}$-definable equivalence relation on~$P(\rho^{m_i})\cap L_{\alpha}$. Then by the Construction Theorem, one can build a model of ${\tt \Sigma^1_1\mbox{-}[E_1, \ldots, E_k]A}+{\tt GC}$ of the following form, wherein again the global well-order~$<$ is given by the membership relation on~$\rho$:
\begin{equation}\label{eqn:thestructureredux}
\mathcal{N}=(\rho, P(\rho)\cap L_{\alpha}, P(\rho\times \rho)\cap L_{\alpha}, \ldots, \partial_{E_1}, \ldots, \partial_{E_k},<)
\end{equation}
Since satisfaction of formulas in the signature of~$\mathcal{M}$ is invariant between it and its expansion~$\mathcal{N}$, this model~$\mathcal{N}$ is the witness to the consistency of~${\tt \Sigma^1_1\mbox{-}[E_1, \ldots, E_k]A}+{\tt SO}_m+{\tt GC}$.
\end{proof}

\section{Identifying the Well-Founded Extensions}\label{sec05}

The goal of this section is to establish the Main Theorem~\ref{thm:main}. This is done in two steps: (i) first by identifying in Theorem~\ref{thM:firstidwef} the well-founded extensions within models induced via the Construction Theorem~\ref{thm:BCT} from~$L_{\alpha}$, and (ii) second in Theorem~\ref{thm1} by an identification within models satisfying~$\mathrm{Axiom\mbox{\;}Beta}$~(cf. Definition~\ref{defna:axiombeta}). The basic idea of these proofs is to relate the notion~$\mathrm{Trcl}_{\eta}(x)$ from \S\ref{sec02.1} equation~(\ref{eqn:iamtransitiveclosureeta}) defined in the object-language of a model of~${\tt \Sigma^1_1\mbox{-}LB}$ to the notion~$\mathrm{trcl}_{\eta}(x)$ defined in the meta-language. In particular, given an arbitrary relation~$R$, the notion~$\mathrm{trcl}_{R}(x)$ is defined to be the set of all~$y$ such that there is a finite sequence~$x_1, \ldots, x_n$ such that~$x_1=y$ and~$x_n=x$ and~$x_{m} R x_{m+1}$ for all~$m< n$. So a model~$\mathcal{N}$ of~${\tt \Sigma^1_1\mbox{-}LB}$ induces a specific relation~$\eta$ via the definition of the Fregean membership relation from equation~(\ref{eqn:Fregemembership}), and then~$\mathrm{trcl}_{\eta}(x)$ is defined to be~$\mathrm{trcl}_{R}(x)$ with~$R=\eta$. Finally, recall that the well-founded extensions~$\mathrm{wfExt}$ were defined in~(\ref{eqnasdfsadf2}).

\begin{thm}\label{thM:firstidwef} (\emph{First Identification of Well-Founded Extensions})
Suppose~$n\geq 1$. Suppose that~$L_{\alpha}$ is~$\Sigma_n$-admissible. Let~$\rho=\rho_n(\alpha)$ and let~$\partial:L_{\alpha}\rightarrow \rho$ be a witnessing~$\utilde{\Sigma}_n^{L_{\alpha}}$-definable injection. Suppose also that~$\rho<\alpha$. Then the structure
\begin{equation}
\mathcal{N} = (\rho, P(\rho)\cap L_{\alpha}, P(\rho\times \rho)\cap L_{\alpha}, \ldots, \partial\upharpoonright (P(\rho)\cap L_{\alpha})) 
\end{equation}
is a model of~${\tt \Sigma^1_1\mbox{-}LB}+{\tt GC}$, where the global well-order on objects is given by the membership relation on~$\rho$. Further:
\begin{align}
\mathrm{wfExt}(\mathcal{N})  = \{x\in &\rho :   (\mathrm{trcl}_{\eta}(x)\cup \{x\}, \eta) \mbox{ is} \notag\\
& ~\mbox{$\utilde{\Delta}^{L_{\alpha}}_n$-well-founded} \; \& \; (\mathrm{trcl}_{\eta}(x)\cup \{x\}) \subseteq \mathrm{rng}(\partial)\} \label{whasdsafsafdsafds}
\end{align}
Moreover, there is a~$\utilde{\Sigma}_n^{L_{\alpha}}$-definable embedding~$j:(L_{\alpha},\in)\rightarrow (\mathrm{wfExt}(\mathcal{N}), \eta)$, and its image is:
\begin{align}
\mathrm{wfExt}_{\ast}(\mathcal{N})  = \{x\in & \rho :   (\mathrm{trcl}_{\eta}(x) \cup \{x\}, \eta) \mbox{ is } \notag \\
& \mbox{well-founded} \; \& \; (\mathrm{trcl}_{\eta}(x) \cup \{x\}) \subseteq \mathrm{rng}(\partial)\}\label{whasdsafsafdsafds123423}
\end{align}
Finally, the isomorphism~$j:(L_{\alpha},\in)\rightarrow (\mathrm{wfExt}_{\ast}(\mathcal{N}), \eta)$ is the inverse of the Mostowski collapse~$\pi: (\mathrm{wfExt}_{\ast}(\mathcal{N}), \eta)\rightarrow (L_{\alpha},\in)$.
\end{thm}
\noindent For the statement of the Mostowski collapse theorem, see the discussion immediately following the definition of Axiom~Beta (Definition~\ref{defna:axiombeta}).
\begin{proof}
By the Construction Theorem~\ref{thm:BCT}, the structure~$\mathcal{N}$ is a model of~${\tt \Sigma^1_1\mbox{-}LB}+{\tt GC}$. Now we argue for the identity in equation~(\ref{whasdsafsafdsafds}). To see this identity, let us first show both of the following, wherein~$x$ is an arbitrary element of~$\rho$:
\begin{align}
&  w\in \mathrm{trcl}_{\eta}(x) \Longrightarrow \mathcal{N}\models (\mathrm{Trcl}_{\eta}(x))(w)\label{help1} \\
&   (\mathrm{trcl}_{\eta}(x)\cup \{x\}) \subseteq \mathrm{rng}(\partial) \Longrightarrow \mathrm{trcl}_{\eta}(x) \in L_{\alpha} \label{help2}
\end{align}
For equation~(\ref{help1}), suppose that~$w\in \mathrm{trcl}_{\eta}(x)$ and suppose that~$F\in (P(\rho)\cap L_{\alpha})$ is such that~$\mathcal{N}\models [\forall \; z \; (z\eta x \rightarrow Fz ) \; \& \; \forall \; u,v \; ((Fv \; \& \; u\eta v) \rightarrow Fu)]$. We must show that~$w\in F$. Since~$w\in \mathrm{trcl}_{\eta}(x)$, choose a sequence~$y_1, \ldots, y_n\in \rho$ such that~$y_1=w$ and~$y_n=x$ and~$y_i\eta y_{i+1}$ for all~$i<n$. Then we may show by induction on~$0<k\leq n\mbox{-}1$ that~$y_{n-k}\in F$.

For equation~(\ref{help2}), first define a map~$\tau: P(\rho)\rightarrow P(\rho)$ by~$\tau(U) = \{v\in \rho: \exists \; w \in U \; v\eta w\}$. Now, it follows from the proposition on the existence of restricted~$\eta$-relation (Proposition~\ref{iamaveryhelpfulprop}) that the map~$\tau$ has the property:
\begin{align}
  [U\in L_{\alpha}  \; \& \; U\subseteq \mathrm{rng}(\partial)] 
  & \Rightarrow     \exists \; S\in (P(\rho\times \rho)\cap L_{\alpha}) \;  [\forall \; w\in U \; \partial(S[w])=w \notag \\
  & \; \& \; \tau(U) = \{v\in \rho: \exists\; w\in U\; v\in S[w]\}\in L_{\alpha}] \label{eqn111}
\end{align}
Let us note one further property of the map~$\tau$, namely its connection to transitive closure:
\begin{equation}\label{eqn:Lobvcios}
U\in (P(\rho)\cap L_{\alpha})\Rightarrow \mathrm{trcl}_{\eta}(\partial(U)) = \bigcup_{n=0}^{\infty} \tau^{(n)} (U)
\end{equation}
To see this, suppose that~$U\in (P(\rho)\cap L_{\alpha})$. First consider the left-to-right direction of the identity. Suppose that~$y\in \mathrm{trcl}_{\eta}(\partial(U))$. Then there are~$y_1, \ldots, y_n$ where~$y_1=y$ and~$y_n=\partial(U)$ and~$y_i\eta y_{i+1}$ for~$i<n$. By induction on~$0<k\leq n\mbox{-}1$ we may then show that~$y_{n\mbox{-}k}\in \tau^{(k\mbox{-}1)}(U)$. Second, consider the right-to-left direction of the identity in equation~(\ref{eqn:Lobvcios}). For this one simply shows by induction on~$n\geq 0$, that~$\tau^{(n)}(U)\subseteq \mathrm{trcl}_{\eta}(\partial(U))$.

Turning now to the verification of equation~(\ref{help2}), suppose that~$(\mathrm{trcl}_{\eta}(x)\cup \{x\}) \subseteq \mathrm{rng}(\partial)$.  Then~$\partial(X)=x$ for some~$X\in (P(\rho)\cap L_{\alpha})$. Now we argue that~$\tau^{(n)}(X)\in L_{\alpha}$ for all~$n\geq 0$. Clearly this holds for~$n=0$, since by hypothesis one has that~$\tau^{(0)}(X)=X\in L_{\alpha}$. Suppose, for the induction step, that~$\tau^{(n)}(X)\in L_{\alpha}$. Then by equation~(\ref{eqn:Lobvcios}) we can collect together the following information: $\tau^{(n)}(X)\in L_{\alpha}$ and $\tau^{(n)}(X)\subseteq \mathrm{trcl}_{\eta}(x) \subseteq  \mathrm{rng}(\partial)$. Then we can deduce immediately from equation~(\ref{eqn111}) that~$\tau^{n+1}(X)=\tau(\tau^{n}(X))\in L_{\alpha}$. So now we have finished arguing that~$\tau^{(n)}(X)\in L_{\alpha}$ for all~$n\geq 0$. By appealing repeatedly to the proposition on the existence of restricted~$\eta$-relation (Proposition~\ref{iamaveryhelpfulprop}),  one has that~$L_{\alpha}$ models that for all~$n<\omega$ there is a sequence~$\langle U_0, S_0, \ldots, U_n, S_n\rangle$ of elements of~$U_i\in P(\rho)\cap L_{\alpha}$,~$S_i\in P(\rho\times\rho)\cap L_{\alpha}$ such that $U_0=X$ and
\begin{align}
&   \forall \; m\leq n \; \forall \; w\in U_m \; \partial(S_m[w])=w\label{manythree1} \\
&  \forall \; m<n \; U_{m+1} = \{v\in \rho: \exists \; w\in U_m \; v\in S_m[w]\}\label{manythree3}
\end{align}
Let~$n<\omega$ and let~$\langle U_0, S_0, \ldots, U_n, S_n\rangle$ be such a sequence. We argue by induction on~$m\leq n$ that~$U_m=\tau^{(m)}(X)$. Clearly this holds for~$m=0$ since $U_0=X$. Suppose it holds for~$m<n$. To see it holds for~$m+1$, note that equation~(\ref{manythree1}) and equation~(\ref{manythree3}) and the induction hypothesis imply 
\begin{align}
& \forall \; w\in \tau^{(m)}(X) \; \partial(S_m[w])=w \\
& U_{m+1} = \{v\in \rho: \exists \; w\in\tau^{(m)}(X) \; v\eta w\} = \tau^{(m+1)}(X) 
\end{align}
So consider the following function~$f:\omega\rightarrow L_{\alpha}$ defined as follows:~$f(m)=U$ iff there is a sequence~$\langle U_0, S_0, \ldots, U_m, S_m\rangle$  satisfying (\ref{manythree1})-(\ref{manythree3}) such that~$U=U_m$. Then the graph of~$f$ is~$\utilde{\Sigma}_n^{L_{\alpha}}$-definable and so by~$\Sigma_n$-replacement, its graph exists as a set in~$L_{\alpha}$. Hence the infinite sequence {}$\langle \tau^{(0)}(X),\tau^{(1)}(X), \ldots, \tau^{(n)}(X), \ldots\rangle$ is an element of~$L_{\alpha}$ and so by equation~(\ref{eqn:Lobvcios}), one also has that~$\mathrm{trcl}_{\eta}(x)=\mathrm{trcl}_{\eta}(\partial(X))\in L_{\alpha}$. So we have finished now the verification of equation~(\ref{help2}).

Now we proceed to the verification of equation~(\ref{whasdsafsafdsafds}). Suppose first that we have an extension~$x\in \mathrm{wfExt}(\mathcal{N})$. Recall that the membership conditions of~$\mathrm{wfExt}(\mathcal{N})$ are defined in equation~(\ref{eqnasdfsadf1}), so that
\begin{equation}\label{eqnsadfsdaf}
\mathcal{N}\models (\mathrm{Trcl}_{\eta}(\sigma(x)), \eta) \mbox{ is well-founded} \; \& \; (\mathrm{Trcl}_{\eta}(\sigma(x))) \subseteq \mathrm{rng}(\partial)
\end{equation}
By equation~(\ref{help1}), we automatically have that
\begin{equation}
(\mathrm{trcl}_{\eta}(x)\cup \{x\}) \subseteq \{w\in \rho: \mathcal{N}\models \mathrm{Trcl}_{\eta}(x)(w)\vee w=x\} \subseteq \mathrm{rng}(\partial)
\end{equation}
Hence from equation~(\ref{help2}), we may conclude that~$\mathrm{trcl}_{\eta}(x)\in L_{\alpha}$. Note that if we set~$F=\mathrm{trcl}_{\eta}(x)$ then~$F$ satisfies the following condition:
\begin{equation}
[\forall \; z \; (z\eta x \rightarrow Fz ) \; \& \; \forall \; u,v \; ((Fv \; \& \; u\eta v) \rightarrow Fu)]\label{dsafasdfasdfasdf}
\end{equation}
Since~$F\in (P(\rho)\cap L_{\alpha})$, it follows that the converse to equation~(\ref{help1}) holds as well, so that we may conclude $\mathrm{trcl}_{\eta}(x) = \{w\in \rho: \mathcal{N}\models \mathrm{Trcl}_{\eta}(x)(w)\}$. So now suppose that~$(\mathrm{trcl}_{\eta}(x) \cup \{x\}, \eta)$ is not~$\utilde{\Delta}^{L_{\alpha}}_n$-well-founded. Then there is some non-empty~$\utilde{\Delta}^{L_{\alpha}}_n$-definable subset~$Z$ of~$(\mathrm{trcl}_{\eta}(x) \cup \{x\}, \eta)$ which has no~$\eta$-least member. By~$\Delta_n$-separation in~$L_{\alpha}$ on the set~$(\mathrm{trcl}_{\eta}(x)\cup \{x\})\in L_{\alpha}$, we have~$Z\in P(\rho)\cap L_{\alpha}$, which is a contradiction. So we just completed the left-to-right direction of equation~(\ref{whasdsafsafdsafds}). For the other direction, suppose that~$x\in \rho$ and
\begin{equation}
(\mathrm{trcl}_{\eta}(x) \cup \{x\}, \eta) \mbox{ is~$\utilde{\Delta}^{L_{\alpha}}_n$-well-founded} \; \& \; (\mathrm{trcl}_{\eta}(x) \cup \{x\}) \subseteq \mathrm{rng}(\partial)
\end{equation}
Then equation~(\ref{help2}) implies that~$\mathrm{trcl}_{\eta}(x)\in L_{\alpha}$. By a similar argument, we have that~$x\in \mathrm{wfExt}(\mathcal{N})$. So we have now finished verifying equation~(\ref{whasdsafsafdsafds}).

Now we turn to constructing an embedding~$j:L_{\alpha}\rightarrow \rho$. By transfinite recursion, there is~$\utilde{\Sigma}_n^{L_{\alpha}}$-definable~$j:L_{\alpha}\rightarrow \rho$ which satisfies~$j(x) = \partial(\{j(y): y\in x\})$. Then one has that~$y \in x$ implies~$j(y)\eta j(x)$. Further, since~$\partial:L_{\alpha}\rightarrow \rho$ is an injection, we may argue by induction that~$j:L_{\alpha}\rightarrow \rho$ is an injection. Since~$j:L_{\alpha}\rightarrow \rho$ is an injection,~$y \in x$ iff~$j(y)\eta j(x)$. Hence,~$j:L_{\alpha}\rightarrow \rho$ is indeed an embedding.

Now we argue that~$j:L_{\alpha}\rightarrow \mathrm{wfExt}_{\ast}(\mathcal{N})$. First let us show:
\begin{equation}\label{nsdafasdfsad}
x\in L_{\alpha} \Longrightarrow  (\mathrm{trcl}_{\eta}(j(x)) \cup \{j(x)\}) \subseteq \mathrm{rng}(j)\subseteq \mathrm{rng}(\partial)
\end{equation} 
Let~$x\in L_{\alpha}$ and let~$y\in \mathrm{trcl}_{\eta}(j(x))$. Then there are 
~$y_1, \ldots, y_n$ in~$\rho$ with~$y_1=y$ and~$y_n=j(x)$ and~$y_1 \eta y_2, \ldots y_{n\mbox{-}1}\eta y_n$. Then using the definition of~$j$  we may argue by induction that~$y_i=j(x_i)$ for~$x_i\in L_{\alpha}$. Let us now argue that 
\begin{equation}\label{nsdafasdfsad2}
x\in L_{\alpha} \Longrightarrow  (\mathrm{trcl}_{\eta}(j(x)) \cup \{j(x)\}, \eta) \mbox{ is well-founded}
\end{equation}
For, suppose that there was an infinite descending~$\eta$-sequence~$y_n$ in the set~$(\mathrm{trcl}_{\eta}(j(x)) \cup \{j(x)\}) \subseteq \mathrm{rng}(j)$. Then since~$j$ is an embedding this would lead to an infinite descending~$\in$-sequence.

Before proceeding, let's note that~$\eta$ is well-founded on~$\mathrm{wfExt}_{\ast}(\mathcal{N})$. For, suppose that~$\emptyset \neq X\subseteq \mathrm{wfExt}_{\ast}(\mathcal{N})$. Choose~$x$ with~$Xx$, so that of course~$x$ is in~$\mathrm{wfExt}_{\ast}(\mathcal{N})$. Then consider~$X^{\prime}=X\cap (\mathrm{trcl}_{\eta}(x) \cup \{x\})$, which is a non-empty subset of~$\mathrm{trcl}_{\eta}(x) \cup \{x\}$. So there is some~$x_0$ with~$X^{\prime}x_0$ such that~$y\eta x_0$ implies~$\neg X^{\prime}y$. Suppose that~$y\eta x_0$ with~$Xy$. Since~$x_0$ is in~$\mathrm{trcl}_{\eta}(x) \cup \{x\}$ and~$y\eta x_0$, we have that~$y$ is in~$(\mathrm{trcl}_{\eta}(x)\cup \{x\})$. Then of course~$y$ is in~$X^{\prime}=X\cap (\mathrm{trcl}_{\eta}(x) \cup \{x\})$, which is a contradiction. So indeed~$\eta$ is well-founded on~$\mathrm{wfExt}_{\ast}(\mathcal{N})$.

Now let us argue that~$j:L_{\alpha}\rightarrow \mathrm{wfExt}_{\ast}(\mathcal{N})$ is surjective. First note that it follows from the definitions that the class~$\mathrm{wfExt}_{\ast}(\mathcal{N})$ is transitive in the following sense:
\begin{equation}\label{asdfasdfadsfdsaf}
[y,z\in \rho \; \& \; y\in \mathrm{wfExt}_{\ast}(\mathcal{N}) \; \& \; z\eta y] \Longrightarrow z\in \mathrm{wfExt}_{\ast}(\mathcal{N})
\end{equation}
So let's proceed in establishing surjectivity by reductio: suppose that~$j:L_{\alpha}\rightarrow \mathrm{wfExt}_{\ast}(\mathcal{N})$ is not surjective. So there is some~$y\in \mathrm{wfExt}_{\ast}(\mathcal{N})\setminus j\mbox{''}L_{\alpha}$. Since~$\eta$ is well-founded on~$\mathrm{wfExt}_{\ast}(\mathcal{N})$ and since~$\mathrm{wfExt}_{\ast}(\mathcal{N})$ is transitive~(\ref{asdfasdfadsfdsaf}), there is~$y\in \mathrm{wfExt}_{\ast}(\mathcal{N})\setminus j\mbox{''}L_{\alpha}$ such that
\begin{equation}
z\eta y \Longrightarrow z\in (\mathrm{wfExt}_{\ast}(\mathcal{N})\cap j\mbox{''}L_{\alpha})
\end{equation}
Since~$y\in \mathrm{wfExt}_{\ast}(\mathcal{N})\subseteq \mathrm{rng}(\partial)$, choose~$Y\in (P(\rho)\cap L_{\alpha})$ such that~$\partial(Y)=y$. Then by the previous equation, we may conclude that $L_{\alpha} \models \forall \; z\in Y \; \exists \; x \; j(x)=z$. By~$\Sigma_n$-collection, choose~$X\in L_{\alpha}$ such that $L_{\alpha} \models \forall \; z\in Y \; \exists \; x\in X \; j(x)=z$. Then set~$X^{\prime} = X\cap j^{-1}(Y)=\{x\in X: j(x)\in Y\}$ which is in~$L_{\alpha}$ by~$\Delta_n$-separation since in addition to its natural~$\utilde{\Sigma}_n^{L_{\alpha}}$-definition it has the following~$\utilde{\Pi}_n^{L_{\alpha}}$-definition: $X^{\prime} = \{x\in X: \forall \; y\in (L_{\alpha}\setminus Y) \; j(x)\neq y\}$. Also~$\{j(x): x\in X^{\prime}\} = Y$, so that we have $j(X^{\prime}) = \partial(\{j(x): x\in X^{\prime}\}) = \partial(Y)=y$ which contradicts the hypothesis that~$y$ was not in the image of~$j$.

Finally note that the isomorphism~$j:(L_{\alpha},\in)\rightarrow (\mathrm{wfExt}_{\ast}(\mathcal{N}), \eta)$ is the inverse of the Mostowski collapse~$\pi: (\mathrm{wfExt}_{\ast}(\mathcal{N}), \eta)\rightarrow (L_{\alpha},\in)$ due to the uniqueness of the latter isomorphism.
\end{proof}

\begin{thm}\label{thm1} (\emph{Second Identification of the Well-Founded Extensions})
Suppose that~$n\geq 1$ and~$L_{\alpha}$ is~$\Sigma_n$-admissible and satisfies Axiom~Beta. Let~$\rho=\rho_n(L_{\alpha})<\alpha$  and let~$\partial:L_{\alpha}\rightarrow \rho$ be a witnessing~$\utilde{\Sigma}_n^{L_{\alpha}}$-definable injection. Then the structure
\begin{equation}
\mathcal{N} = (\rho, P(\rho)\cap L_{\alpha}, P(\rho\times \rho)\cap L_{\alpha}, \ldots, \partial\upharpoonright P(\rho)\cap L_{\alpha}) 
\end{equation}
is a model of~${\tt \Sigma^1_1\mbox{-}LB}+{\tt GC}$, where the global well-order on objects is given by the membership relation on~$\rho$. Further,~$(L_{\alpha}, \in)$ is isomorphic to~$(\mathrm{wfExt}(\mathcal{N}),\eta)$.
\end{thm}
\begin{proof}
By the previous theorem, it suffices to show that~$\mathrm{wfExt}(\mathcal{N})\subseteq \mathrm{wfExt}_{\ast}(\mathcal{N})$. For this, it suffices to show that for all~$x\in \rho$ we have
\begin{align}
  [(\mathrm{trcl}_{\eta}(x) \cup \{x\}, \eta) &  \mbox{ is~$\utilde{\Delta}^{L_{\alpha}}_n$-well-founded} \; \& \; (\mathrm{trcl}_{\eta}(x) \cup \{x\}) \subseteq \mathrm{rng}(\partial)] \notag\\
& \Longrightarrow    (\mathrm{trcl}_{\eta}(x) \cup \{x\}, \eta) \mbox{ is well-founded} \label{eqn:13412341234}
\end{align} 
So suppose that~$x\in \rho$ satisfies the hypothesis of this conditional. Then define the set~$X=(\mathrm{trcl}_{\eta}(x) \cup \{x\})$, which is in~$L_{\alpha}$ by (\ref{help2}) of the previous proof. Then by the proposition on the existence of restricted~$\eta$-relation (Proposition~\ref{iamaveryhelpfulprop}), choose binary relation~$E_X\in L_{\alpha}$ such that~$E_X\subseteq \rho\times X$ and such that~$Xa$ implies~$E_X(b,a)$ iff~$b\eta a$. Since~$X$ is~$\eta$-transitive, we have that~$E_X\subseteq X\times X$. Then the hypothesis that~$(\mathrm{trcl}_{\eta}(x) \cup \{x\}, \eta)$ is~$\utilde{\Delta}^{L_{\alpha}}_n$-well-founded and the~$\eta$-transitivity of~$X$ implies
\begin{equation}
L_{\alpha} \models (X,E_X) \mbox{ is well-founded and extensional}
\end{equation}
Since the structure~$L_{\alpha}$ satisfies~$\mathrm{Axiom\mbox{\;}Beta}$, the structure~$L_{\alpha}$ satisfies the Mostowski Collapse Theorem (cf. discussion following Definition~\ref{defna:axiombeta}). Then there is a transitive set~$M$ in~$L_{\alpha}$ and a map~$\pi$ in~$L_{\alpha}$ such that~$\pi:(X,E_X)\rightarrow (M,\in)$ is an isomorphism. Suppose that~$(X,E_X)$ is not well-founded. Then there is an infinite decreasing~$\eta$-sequence~$x_i$ in~$X\subseteq L_{\alpha}$. Then~$\pi(x_i)$ is an infinite decreasing~$\in$-sequence.
\end{proof}

This allows us to now establish the Main Theorem~\ref{thm:main}:
\begin{proof} (of Theorem~\ref{thm:main}): By compactness, this follows from the Existence Theorem~\ref{thm:basicexistence} and the Second Identification of the Well-Founded Sets Theorem~\ref{thm1}. Here we're also appealing to the connection between the union of the axiomatic characterizations of~$\Sigma_n$-admissibility and~${\tt ZFC\mbox{-}P}$, which we noted immediately after the definition of~$\Sigma_n$-admissibility (cf. Definition~\ref{defn:nadmissible}). 
\end{proof}

\subsection*{Acknowledgements}
I was lucky enough to be able to present parts of this work at a number of workshops and conferences, and I would like to thank the participants and organizers of these events for these opportunities. I would like to especially thank the following people for the comments and feedback which I received on these and other occasions: Robert~Black, Roy~Cook, Matthew~Davidson, Walter~Dean, Marie~Du\v{z}\'{\i}, Kenny~Easwaran, Fernando~Ferreira, Martin~Fischer, Rohan~French, Salvatore~Florio, Kentaro~Fujimoto, Jeremy~Heis, Joel~David~Hamkins, Volker~Halbach, Ole Thomassen~Hjortland, Luca~Incurvati, Daniel~Isaacson, J\"onne~Kriener, Graham~Leach-Krouse, Hannes~Leitgeb, {\O}ystein~Linnebo, Paolo~Mancosu, Richard~Mendelsohn, Tony~Martin, Yiannis~Moschovakis, John~Mumma, Pavel~Pudl\'ak, Sam~Roberts, Marcus~Rossberg, Tony~Roy, Gil~Sagi, Florian~Steinberger, Iulian~Toader, Gabriel~Uzquiano, Albert~Visser, Kai~Wehmeier, Philip~Welch, Trevor~Wilson, and Martin~Zeman. This paper has likewise been substantially bettered by the feedback and comments of the editors and referees of this journal, to whom I express my gratitude. While composing this paper, I was supported by a Kurt G\"odel Society Research Prize Fellowship and by {\O}ystein Linnebo's European Research Council funded project ``Plurals, Predicates, and Paradox.'' 

\bibliographystyle{plain}
\bibliography{mybib.v011}

\end{document}